\documentclass[preprint,10pt]{elsarticle}
\usepackage{amsmath}
\usepackage{amssymb}
\usepackage{amsthm}
\usepackage{mathrsfs}
\usepackage{amsfonts} %%\mathbb{R} \mathcal \mathscr
\usepackage{CJK}
\usepackage{multicol}
\usepackage{multirow}
\usepackage{diagbox}
\usepackage{graphicx}
\usepackage{graphics}
\usepackage{subfigure}
\usepackage{lineno} %显示行号
\usepackage{caption}

\usepackage[font=scriptsize]{subfig} %更改图片子标题大小
\allowdisplaybreaks % allow the equation display in different pages
\usepackage{epstopdf} % transfer the figure with eps formation to figure with pdf formation
\usepackage{pgf,tikz}
\usepackage{tikz}
\usetikzlibrary{arrows,shapes,snakes,shadows,positioning,automata,patterns}
\usetikzlibrary{trees,decorations.pathmorphing,decorations.markings}
\usetikzlibrary{backgrounds}

\usepackage{ifpdf}
\usepackage{epstopdf}
\usepackage{cases}
\usepackage{array}
\usepackage{textcomp}
\usepackage{stfloats}
\usepackage{url}
\usepackage{verbatim}

\usepackage{color} %\textup{sign}
\usepackage{float}
\usepackage{bm}

\newtheorem{rem}{Remark}

\newtheorem{thm}{Theorem}
\newtheorem{cor}{Corollary}
\newtheorem{lem}{Lemma}
\newtheorem{ass}{Assumption}
\newtheorem{defn}{Definition}
\newtheorem{example}{Example}

\newcommand{\R}{\mathbb{R}}

\newcommand{\F}{\mathscr{F}}

\newcommand{\T}{\top}
\newcommand{\define}{:=}
\newcommand{\ve}[1]{\bm{#1}}
\newcommand{\E}{\ensuremath{\mathrm{E}}} %数学期望E

\DeclareMathOperator*{\col}{col}

\DeclareMathOperator*{\dist}{dist}%distance

\usepackage[linesnumbered,ruled,vlined]{algorithm2e}%[ruled,vlined]{
\usepackage{algpseudocode}
\journal{Automatica}

\begin{document}
\captionsetup{font={small}}
	
\begin{frontmatter}

\title{Social optimization in noncooperative games under high-level regulation\tnoteref{label1}}

\author[1]{Kaixin Du}\ead{dukx@tongji.edu.cn}   % Add the
\author[1,2,3,4]{Min Meng\corref{label2}}\ead{mengmin@tongji.edu.cn}               % e-mail address
\author[5]{Xiaoming Hu}\ead{hu@kth.se}  % (ead) as shown

\address[1]{Shanghai Research Institute for Intelligent Autonomous Systems, Tongji University, Shanghai 201210, P. R. China}  % Please supply
\address[2]{Department of Control Science and Engineering, College of Electronics and Information Engineering, Tongji University, Shanghai 201804, P. R. China}
\address[3]{National Key Laboratory of Autonomous Intelligent Unmanned Systems, Shanghai, P.
R. China}          % full addresses
\address[4]{Frontiers Science Center for Intelligent Autonomous Systems, Ministry of Education, Shanghai, P. R. China}
\address[5]{Optimization and Systems Theory, KTH Royal Institute of Technology, Stockholm 10044, Sweden}
\tnotetext[label1]{This work was partially supported by National Science and Technology Major Project under grant 2022ZD0119702, the National Natural Science Foundation of China under grant 62473293 and 62088101, Shanghai Municipal Science and Technology Major Project under grant 2021SHZDZX0100. ({\em Corresponding author: Min Meng}).  }
\cortext[label2]{Corresponding author.}

\begin{abstract}
Motivated by the increasing attention to overall social benefits in networked multi-agent systems, this paper investigates an optimization problem building on noncooperative games under high-level  regulation, which can be formulated in a bilevel structure.
Specifically,  the low level consists of a noncooperative game, where each player competes to minimize its own  cost function that depends not only on the strategies of all players, but also on an intervention decision of a regulator located at the high level. Under the intervention of the high-level regulator,  the low-level players aim to seek a Nash equilibrium (NE), which indeed is related to the regulator's decision. Meanwhile, the regulator in the high level  attempts to  achieve the social optimum, that is, to minimize the sum of all players' costs obtained at the NE. This bilevel social optimization problem is proven to be nonconvex and nonsmooth, leading to challenges for solving it effectively, as the exact gradient of cost sum functions may not be available. To address this intricate problem, an inexact zeroth-order algorithm is developed by virtue of the smoothing techniques, allowing for approximating the NE of the low-level game and thus estimating the required gradients. It is rigorously shown that the devised algorithm achieves a  sublinear convergence rate for computing an approximate stationary point of the studied problem.  %it is rigorously shown that the devised algorithm achieves a  sublinear convergence rate for computing a stationary point of a related optimization problem with a smoothed  objective.
Moreover, the sublinear convergence rate in the scenario where the exact equilibrium of the low-level game is available is established. Finally, numerical simulations are conducted to demonstrate the efficiency of theoretical findings.
\end{abstract}


\end{frontmatter}

\begin{multicols}{2}

\section{Introduction}
Game theory provides a suitable framework for modeling strategic interactions among rational players that aim  to selfishly minimize their private, yet interdependent objective functions. This  competitive setting has been commonly applied to various domains, including  power allocation \cite{sun25Learning}, network congestion control \cite{barrera2014dynamic}, imitation and reinforcement learning \cite{Hierarchical23}, and social networks \cite{mani21quan}. One of most important issues in noncooperative games is to seek Nash equilibria (NEs) or generalized Nash equilibria (GNEs) which eventually dictate stable and satisfactory states wherein no player possesses an incentive for unilateral deviation.% within the {\color{blue}strategic environment.}
%social networks \cite{ghaderi2014opinion} %demand-side management in electricity market \cite{sun25Learning}

%unconstrained,\cite{Geometric20}
In the past decades, noncooperative games have been extensively explored under different scenarios. Generally speaking, the existing algorithms for seeking NEs of noncooperative games can be grouped into three categories. The simplest setup pertains to unconstrained games \cite{Nguyen23}, wherein each player possesses the liberty to select strategies from the entire real number or vector spaces. Moving towards a slightly more intricate case, games with local uncoupled decision set constraints are considered \cite{FengRR24,Nguyen25,MENG2023110919,Bianchi21}, necessitating players to confine their decisions within predetermined nonempty sets, which are often presumed to be closed and convex. Furthermore, much attention is paid for games involving coupled equality or inequality constraints \cite{meng2022linear,Gadjov21,meng24online, Carnevale2022Tracking-based}. An alternative perspective categorizes methods for seeking NE or GNE into two groups, according to the way how information is exchanged. Specifically, one class is the full-decision information scenario, where each player has access to the decisions taken by all others \cite{belgioioso2018,lei22,Persis20}. Unfortunately, in such a scenario, a central coordinator is typically required to broadcast the data to the network, which is impractical in numerous real-world applications \cite{belgioioso2018,ghaderi2014opinion}. The other case that has gained attention recently is the partial-decision information scenario, where each player is capable of estimating decisions of others based only on local information exchanges with neighbors through a communication network \cite{Nguyen23,MENG2023110919,FengRR24,meng2022linear,Gadjov21,Bianchi21,Nguyen25,meng24online,Carnevale2022Tracking-based}, resulting in fully  distributed NE seeking algorithms.
%Additionally, it is noteworthy to highlight the consensus mechanism, leveraged as a powerful tool in noncooperative games to ensure each player's estimates eventually towards consensus \cite{Bianchi21,meng2022linear}.

It should be noted that in noncooperative games, each player acts selfishly and only concerns its individual interest. However, such self-centered behaviors may lead to deterioration in performance of the social cost, that is, the sum of  cost functions of all players. Specially, some existing studies have yielded valuable insights into the detrimental consequences of selfish behaviors on performance degradation in diverse fields, such as supply chains \cite{ye16chain}, resource allocation \cite{johari05}, federated edge learning \cite{pan25bilateral}, and congestion games \cite{Roughgarden02,apt14} by virtue of the metrics of the price of anarchy and the price of stability \cite{Cominetti24,Anshelevich08,Chandan24metho}. A prospective direction of research involves improving social cost while also accommodating the preferences of individual players in noncooperative games. Consequently, a high-level regulator has its purpose to regulate its decision to eusure that the social cost is minimum. For instance, in the context of economic investment problems within network games, the central bank, who wishes to ensure market stability and avoid systemic risk, maximizes the overall social return for all participants by meticulously regulating individual standalone marginal return \cite{Shakarami23}. Similarly, in power systems, the total social resource consumption is minimized through the strategic adjustment of electricity prices \cite{Ma14,sun25Learning}.

To date, there has been some literature on adaptive incentive mechanisms for steering NEs towards a desirable point \cite{Galeotti2020,alpcan09,Maheshwari2021DynamicTF,ratliff21,Maheshwari22}. However, these works were explored from the systems and  control theoretic perspective, and paid attention to sufficient conditions or direct computation of the best targeting intervention for problems exhibiting specific structures.
Moreover, incentive designs for Stackelberg games were investigated in \cite{bacsar2024incentive,Shen2007,sanjari2025incentive}, where leaders attempt to shape the behavior of players by providing incentives. Notably, the leader in such settings are self-interested, as it pursues only its own objective and the optimal performance is ultimately attained from the leader' viewpoint, simultaneously leading to an NE among players. This motivation is often driven by personal interests rather than by the concern of  social benifits.
Recently, combining with the gradient dynamics, the authors in \cite{Shakarami23} studied continuous-time intervention protocols based on limited knowledge, which were rigorously shown to successfully drive actions of players in quadratic network games to the social optimum. For social optimization in nonconvex cooperative aggregative games, a distributed stochastic annealing algorithm was developed in \cite{wang22} without external intervention and was proved to weakly converge. Note that both of \cite{Shakarami23} and \cite{wang22} have not yet derived any result regarding the convergence rate. Furthermore, most of the aforementioned works did not consider games with set constraints and intervention constraints.  In \cite{liu22}, a bilevel incentive design algorithm is devised for probability simplex constrained games and the sublinear  convergence rate  is provided.

It is emphasized that the exploration of social optima in noncooperative games is toward the  global objective of a social optimization problem, thereby facilitating the high-level regulator in further improving resource allocation efficiency. From a theoretical viewpoint, when each player's  cost  function depends on its own strategy and those of other players, as well as a decision determined by the high-level regulator, the resulting situation extends beyond the scope of traditional noncooperative games, shifting the focus towards attaining a global optimal objective. The problem becomes increasingly complex and  challenging due to its inherent bilevel structure, particularly, the mutual coupling of decision-making  process between players and the high-level regulator. Therefore, achieving the social optimum in general noncooperative games remains an open and challenging issue.

%address the conundrum
Motivated by the above analysis, to overcome the obstacle that the behaviors of selfish players proposing to minimize their own costs may lead to the social cost not optimal, this paper investigates the social optimization problem, which inherently features a bilevel structure. In the low level, players participate in a  noncooperative game, where each individual cost function relies not only on all players' strategies but also on an intervention decision of a high-level regulator. Under the intervention of the regulator, all players endeavor to seek the NE corresponding to the regulator's decision. Then, in the high level, the  regulator tries to find an optimal decision so as to achieve the social optimum. The significant challenges arise from the lack of a closed-form solution for the low-level game, along with the inherent nondifferentiability and nonconvexity of the high-level function. An inexact gradient-free method by utilizing smoothing techniques is devised to solve this problem and elaborately analyzed to exhibit sublinear convergence rate. The main contributions are summarized as follows.
\begin{itemize}
  \item [i)]A new social optimization framework based on noncooperative games is formulated which is inherently equipped with a bilevel structure, where the strategies of low-level players and the decisions of the high-level regulator are intricately and mutually coupled. Although the related problems regarding social optimization were investigated in \cite{wang22,Shakarami23,liu22}, the problem settings are rather different. Specifically, there was no intervention involved in \cite{wang22}. In fact, players are reluctant to sacrifice their personal wellbeing to move toward the social optimum without external intervention. Moreover, to design  interventions for special quadratic network games \cite{Shakarami23}, the regulator is presumed to be aware of the social optimum in advance. However, this is often difficult or even impossible in reality.  Therefore, our concerned problem is more general and practical.

      %{\color{red}Different from \cite{Shakarami23}, the games considered here is more general rather than special quadratic games. Additionally, \cite{liu22} is limited to special probability simplex constrained games and \cite{wang22} did not take any constraints into account. However, our problem is explored under both strategy set constraints and intervention constraints simultaneously. In addition,  the studied problem setting is different from them.}
      \item [ii)] The formulated problem is indeed constrained, nonsmooth, and nonconvex, leading to difficulties in problem solving.  This is also the exact gradient of cost sum functions may not be available.  By  employing the randomized smoothness and Moreau smoothness techniques, the studied problem is transformed into a corresponding smoothed unconstrained optimization problem in form.  Then a zeroth-order algorithm is proposed, which enables the low-level game to be solved in an inexact manner to contend with the unavailable exact equilibrium of the low-level game. The devised algorithm utilizes random gradient-free oracles instead of relying on direct gradient information, and does not require the exact equilibrium of the low-level game. As a result, the algorithm becomes more practical and flexible in scenarios where obtaining gradients of the objective functions is infeasible or the exact NE is prohibitively expensive to evaluate.
          % in contrast to previous researches that address only unconstrained cases \cite{}.
      \item [iii)] Through rigorous analysis, in the case where the high-level regulator cannot obtain the NE of the low-level game, %the sublinear convergence rate is established for computing a stationary point of the corresponding smoothed unconstrained optimization problem within the inexact regime.
           the sublinear convergence rate is established for computing  an approximate  stationary point of the concerned problem within the inexact regime.
          Furthermore, the sublinear rate is also derived in the scenario where the exact NE of the low-level game is available. However, the convergence rate of algorithms for social optimization in \cite{Shakarami23,wang22} was not established.
\end{itemize}

The rest of this paper is outlined as follows.  Section \ref{sec2} formulates the problem and introduces the  preliminaries. A zeroth-order algorithm is provided in Section \ref{sec3}, followed by the convergence results. In Section \ref{sec4}, numerical simulations are given to illustrate the theoretical findings. Section \ref{sec5} makes a brief conclusion.

\emph{Notations:} $\R^n$ and $\R^{m\times n}$ denote the sets of $n$-dimensional real column vectors and $m\times n$ real matrices, respectively, endowed with Euclidean norm $\|\cdot\|$. Denote by $M^{\T}$ the transpose of a matrix or vector $M$. $I_N\in\R^{N\times N}$ represents the $N\times N$ identity matrix.  $\ve{1}_N:=(1\ 1 \ldots 1)^{\T}\in\R^N$. $\col\{x_i\}_{i=1}^N:=(x_1^{\T},\ldots,x_N^{\T})^{\T}$. For an integer $m>0$, $[m]:=\{1,\ldots,m\}$. $\otimes$ denotes the Kronecker product. For two matrices $P,Q$, $P\succ Q$ represents $P-Q$ is positive definite. $\lceil a\rceil$ is the smallest integer greater than or equal to scalar $a$. Denote by $\E[\cdot]$ the expectation of a random variable. Given a closed convex set $C\subseteq\R^n$, $\Pi_C(x):=\arg\min_{y\in C}\|x-y\|$ stands for the projection of $x\in\R^n$ onto $C$, $\dist(x,C):=\min_{y\in C}\|x-y\|$ denotes the distance of $x$ to $C$, and  $\mathrm{N}_C(x):=\{s\in\R^n: s^{\T}(y-x)\leq 0, \forall y\in C\}$ represents the normal cone of $x\in C$ at $C$.  For real-valued sequences $\{a_k\}$ and $\{b_k\}$, $a_k = \mathcal{O}(b_k)$ if there exists $C>0$ satisfying $a_k \leq C b_k$. Let $f:\R^n\to(-\infty,\infty]$ be a proper function, the Clarke generalized gradient of $f$ at $x$ is denoted as $\partial f(x) :=\{w \in \mathbb{R}^n: f^{\circ}(x, d) \geq w^{\T}d,  \forall d \in \mathbb{R}^n\}$, where $f^{\circ}(x, d):= \limsup _{z \to x, t \downarrow 0}(f(z+t d)-f(z)) / t$ is the generalized directional derivative of $f$ at $x$ in a given direction $d$. Let $\mathbb{S}^n:=\{u\in\R^n:\|u\|=1$\} and $\mathbb{B}^n:=\{\nu\in\R^n: \|\nu\|\leq 1\}$ be $n$-dimensional unit ball and its surface, respectively. $\sigma_{N-1} (A)$ denotes the second largest singular value of matrix $A\in\R^{N\times N}$. For a function $f=(f_1,f_2,\ldots,f_d)$ with $f_i:\mathbb{R}^n\to\mathbb{R}$ for $i=1,2,\ldots,n$, denote $\nabla f(x)\define [\nabla f_1(x),\ldots,\nabla f_d(x)]$.

\section{Problem formulation and preliminaries}\label{sec2}
In this section, the studied problem is first introduced and then some essential preliminaries related to smoothing techniques are presented.

\subsection{Problem formulation}
This paper introduces a novel social optimization framework grounded in noncooperative games consisting of $N$ players located at  low level and a  regulator in the high level. In the low level, each  player performs in a noncooperative game and aims to selfishly minimize its own cost function, dependent on not only the strategies of all players but also an adjustable intervention decision of the high-level regulator. After the high-level regulator makes a decision first, all players response at once simultaneously, aiming to seek the NE, which indeed is related to the regulator's decision. Then the  regulator in the high level seeks to determine an optimal intervention decision to reach the social optimum. Specifically, the schematic sketch of the problem setup is presented in Fig. \ref{fig:pro}, and the further expressions are detailed below.

%\begin{figure}[htbp]
%\centering
%\includegraphics[width=8.3cm,height=6.5cm]{pro.png}
%\caption{Schematic sketch of the problem setup.\label{fig:pro}}
%\end{figure}

%\begin{figure}[htbp]
%\centering
%\includegraphics[width=8.7cm,height=5.7cm]{prob.png}
%\caption{Schematic sketch of the problem setup.\label{fig:pro}}
%\end{figure}
\begin{figure}[H]
\centering
\includegraphics[width=8.7cm,height=5.7cm]{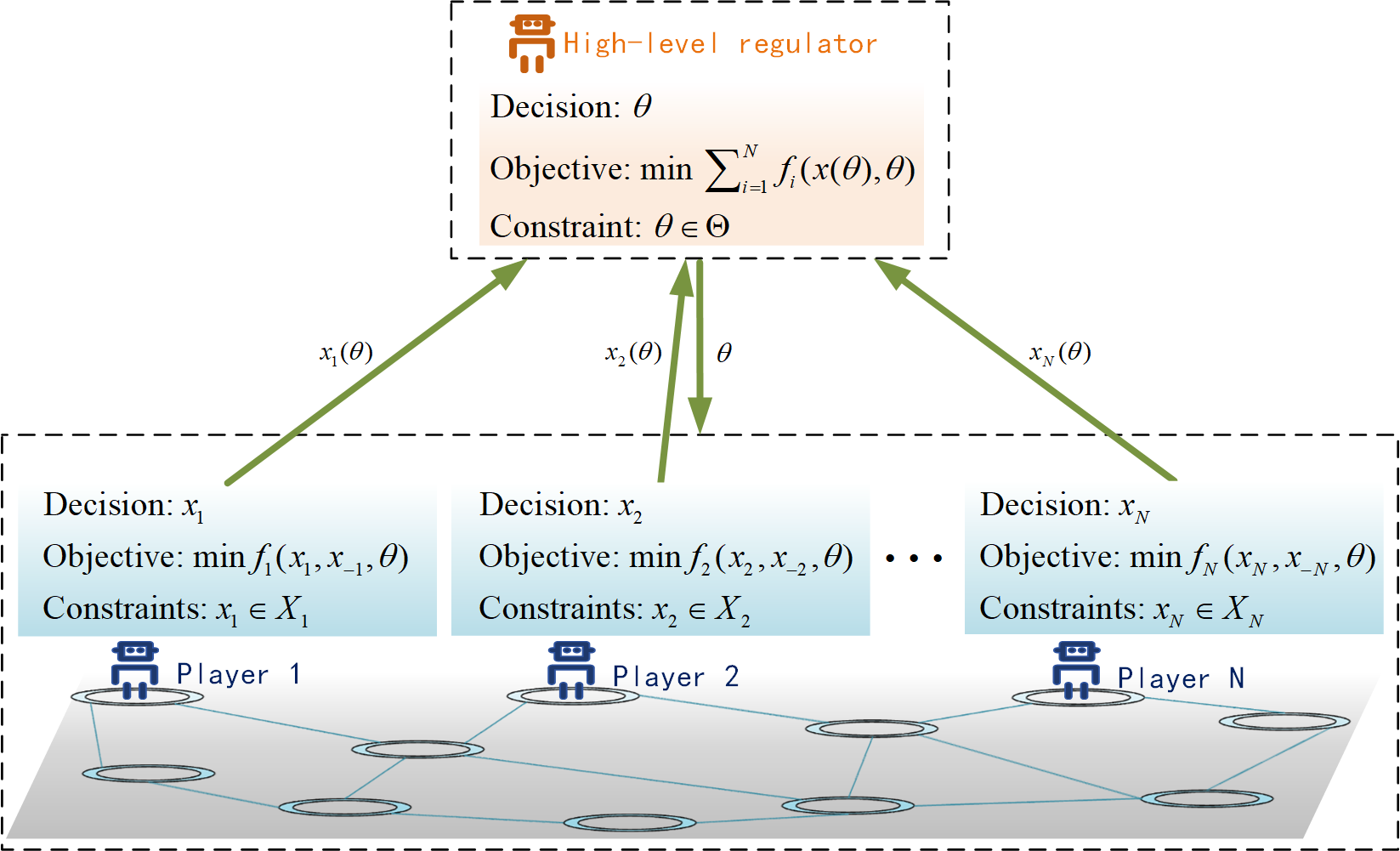}
\caption{Sketch of the problem setup. Each of the $N$ players in the low level aims to minimize its personal cost function $f_i$ given intervention decision $\theta$ of the high-level regulator. Then, the regulator in the high level utilizes required data received from all players to choose an optimal intervention decision to achieve social optimum. Here, $x_{-i}:=\col\{x_i\}_{i\in[N]\backslash\{i\}}$ and $x(\theta):=\col\{x_i(\theta)\}_{i\in[N]}$   represents the NE of the low-level game.} \label{fig:pro}
\end{figure}

\noindent $\textbf{Low-level game.}$ Each player $i\in[N]$ takes strategy $x_i$ from its strategy set $X_i\subseteq\R^{d_i}$. Let $x=\col\{x_i\}_{i\in[N]}\in X:=X_1\times\cdots\times X_N\subseteq\R^d$ with $d=\sum_{i=1}^Nd_i$ and $x_{-i}=\col\{x_i\}_{i\in[N]\backslash\{i\}}\in {X_{-i}}:=X_1\times\cdots\times X_{i-1}\times X_{i+1}\times\cdots \times X_N\subseteq\R^{d-d_i}$ denote, respectively, the joint strategy  of all players and the joint strategy of all other players except $i$. The cost function of player $i$ is $f_i(x_i,x_{-i},\theta)$ which depends on $x_i$, $x_{-i}$ and the high-level regulator's intervention decision $\theta$ selected from a feasible set $\Theta\subseteq\R^n$. Herein, each player is not willing to share its own strategy and objective function with other players for the reason of privacy, and instead only has access to the information of partial players through a communication graph. Given a current decision $\theta$ made by the high-level regulator, players select their strategies with the aim of minimizing individual objective functions $f_i(x_i,x_{-i},\theta)$, $i\in[N]$, to seek the NE that is in their best interest. This leads to a collection of interdependent optimization problems:
\begin{align}\label{pro:game}
\min_{x_i\in X_i} f_i(x_i,x_{-i},\theta), \forall i\in[N].
\end{align}
%In general, the low-level game could admit multiple NE for a particular parametrization.
In this study, we limit our focus to the case where the low-level game (\ref{pro:game}) admits a unique NE for any particular parametrization $\theta\in\Theta$. Upon completion of the game process, an NE $x(\theta)=(x(\theta),x_{-i}(\theta))$ is obtained and defined in the following, which corresponds to a set of costs  of players:  $f_i(x_i(\theta),x_{-i}(\theta),\theta), i\in[N]$.

%$\Gamma(\mathcal{V},f,X;\theta)$
\begin{defn}\label{defn:NE}
For any $\theta\in\R^n$, a strategy profile $x(\theta)=(x_i(\theta),x_{-i}(\theta))\in X$ is an NE for game (\ref{pro:game}) if for every player $i\in[N]$, it holds that
\begin{align}\label{def:NE}
f_i(x_i(\theta),x_{-i}(\theta),\theta)\leq f_i(x_i,x_{-i}(\theta),\theta), \forall x_i\in X_i.
\end{align}
\end{defn}

\noindent $\textbf{High-level optimization.}$  After receiving all players' feedback $x_i(\theta)$, $i\in[N]$, where $x(\theta)$ denotes the NE of the low-level game \eqref{defn:NE},  the high-level regulator endeavors to determine the optimal decision $\theta^*$ that minimizes the sum of cost functions of all players, thereby achieving the social optimum.  %More precisely, the regulator aims to solve the following optimization problem:
%\begin{align}\label{pro:opt}
%\min_{\theta\in\Theta} f(x(\theta),\theta):=\sum_{i=1}^{N}f_i(x(\theta),\theta).
%\end{align}

Although  the information on the players is available to the high-level regulator, the decisions of the  regulator dynamically  depend on the NE of the low-level game, resulting in high complexity and strong coupling. The goal of this paper is to address the following bilevel social optimization problem
\begin{align}\label{prob}
	&\min_{\theta\in\Theta}\sum_{i=1}^N f_i(x_i(\theta),x_{-i}(\theta),\theta),\notag\\
	&\text{s.t.}~f_i(x_i(\theta),x_{-i}(\theta),\theta)\leq f_i(x_i,x_{-i}(\theta),\theta),\notag\\
	&\quad\quad\quad\quad\quad\quad\quad\quad\quad\quad\ \forall i\in[N], x_i\in X_i.
\end{align}

The studied bilevel social optimization probem is inherently nonconvex and nonsmooth, even when each $f_i$  is continuously differentiable and (strong) convex with respect to all its arguments, demonstrated by the following example.
\begin{example}\label{exam-nondiff}
	Consider a noncooperative game with $2$ players. The cost function of player $i$ is $f_i(x_i,x_{-i},\theta)=x_i^2-2x_{-i}-2x_i\theta$ and the strategy set is $X_i=[2/3,1]$, $\forall i\in[2]$. The feasible set of the high-level regulator is $\Theta=[0,1]$.
	Based on the optimal condition (\ref{optcon}), it can be derived that
	\begin{align*}
		x_i(\theta)=
		\left\{
		\begin{aligned}
			&2/3,  & \text{if}\ \theta\in[0,2/3], \\
			&\theta,      & \text{if}\ \theta\in[2/3,1].
		\end{aligned}
		\right.
	\end{align*}
	Through simple calculation, one has
	\begin{align*}
		f(x(\theta),\theta)&=\sum_{i=1}^2f_i(x(\theta),\theta)\\
		&=
		\left\{
		\begin{aligned}
			&-\frac{4}{3}\theta-\frac{8}{9},  & \text{if}\ \theta\in[0,2/3], \\
			&-2\theta^2-4\theta,      & \text{if}\ \theta\in[2/3,1].
		\end{aligned}
		\right.
	\end{align*}
	Obviously, at the point $\theta=2/3$ the mapping $x(\theta)$ and thus $f(x(\theta),\theta)$ are not differentiable.
	In addition, when $\theta\in[2/3,1]$, $f(x(\theta),\theta)=-2\theta^2-4\theta$ is nonconvex. %In conclusion, the corresponding social optimization problem is constrained, nonsmooth and nonconvex.
\end{example}

\begin{rem}
%It seems that the idea of the social optimization (\ref{prob}) appears to be somewhat analogous to Stackelberg games studied in \cite{Maljkovic2024OnDC,linon24,jo2023}, where players in the low-level strive to optimize personal cost functions parametrized by the leaders' decisions and influenced by the decisions of other followers. However, a key distinction lies in the role of the high-level leader.
Note that the key distinction between the social optimization (\ref{prob}) and Stackelberg games studied in \cite{Maljkovic2024OnDC,linon24,bacsar2024incentive,Shen2007,sanjari2025incentive} lies in the role of the high-level leader. Specifically, in the realm of Stackelberg games, the leader pursues the minimization of its own objective while adhering to the equilibrium constraints established by the low-level game among the followers. However, the high-level regulator in social optimization is tasked with optimizing the social cost. Moreover, in  noncooperative games, only players compete with each other to minimize their own cost functions and no other regulators are involved.
Actually, both leaders and players in Stackelberg games and noncooperative games are selfish, as they primarily focus on self-interested behaviors.
%In fact, both players and leaders within Stackelberg games and noncooperative games are actually selfish.
Therefore, the studied problem is essentially different from these frameworks since it is driven by the overall benefit to society, with a high-level regulator coordinating the whole network to reach this goal.
\end{rem}

Some standard assumptions are listed next, which are commonly made in existing literature \cite{meng2022linear,Nguyen23,FengRR24,qu24online,Bianchi21,Carnevale2022Tracking-based}.
\begin{ass}\label{ass:set}
$~~~~~~~~~~$
\begin{itemize}
 \item [i)] The set $X_i$ for any $i\in[N]$ is nonempty, compact and convex.
 \item [ii)] The set $\Theta$ is nonempty, closed, and convex.
\item [iii)] For any $i\in[N]$,  $f_i(x_i,x_{-i},\theta)$ is continuously differentiable and convex with respect to $x_i$ for any $x_{-i}\in \R^{d-d_i}$ and $\theta\in\R^n$.
  \item [iv)]  For any  $i\in[N]$, $f_i(x,\theta)$ is $L_\theta$-Lipschitz continuous with respect to $\theta$ for any $x\in X$, i.e., for any $\theta, \theta'\in\R^n$, there exists a constant $L_\theta>0$ such that
      \begin{align}\label{F-lips-3}
      \|f_i(x,\theta)-f_i(x,\theta')\|\leq L_\theta\|\theta-\theta'\|.
     \end{align}
\end{itemize}
\end{ass}
From Assumption \ref{ass:set}-i) and iii), it can be further obtained that for any $i\in[N]$, $\theta\in\R^n$, and $x, x'\in X$,  there exist positive constants $B_X$ and $L_x$ such that
\begin{align}
&\|x\|\leq B_X,\label{F-lips-1}\\
&\|f_i(x,\theta)-f_i(x',\theta)\|\leq L_x\|x-x'\|\label{F-lips-2}.
\end{align}

Given $\theta\in\R^n$, the pseudo-gradient mapping $G: \R^d\times \R^n\to\R^d$ is defined to be
\begin{align}\label{def:PG}
G(x,\theta)=\col\{\nabla_if_i(x_i,x_{-i},\theta)\}_{i\in[N]}.
\end{align}

\begin{ass}\label{ass:pseudo-gradient}
$~~~~~~~~~~$
\begin{itemize}
\item [i)] The pseudo-gradient $G$ is $\mu$-strongly monotone ($\mu>0$) and $l$-Lipschitz continuous ($l>0$) with respect to $x$ for any $\theta\in\R^n$, i.e., for any $x, x'\in \R^d$,
\begin{align}
&(G(x,\theta)-G(x',\theta))^{\T}(x-x')\geq\mu\|x-x'\|^2,\label{ass:F-1}\\
&\|G(x,\theta)-G(x',\theta)\|\leq l\|x-x'\|\label{ass:F-2}.
\end{align}
\item [ii)] $G$ is $l_\theta$-Lipschitz continuous ($l_\theta>0$) with respect to $\theta$ for any $x\in X$, i.e., for any $\theta, \theta'\in \R^n$,
\begin{align}\label{ass:F-3}
\|G(x,\theta)-G(x,\theta')\|\leq l_\theta\|\theta-\theta'\|.
\end{align}
\end{itemize}
\end{ass}

Under Assumption \ref{ass:set}-i) and iii), it follows from Proposition 1.4.2 in \cite{facchinei2003finite} that for any given $\theta\in\R^n$, a solution $x(\theta)\in X$ is an  NE of game (\ref{pro:game}) if and only if
\begin{align}\label{optcon}
{\bf 0}\in G(x(\theta),\theta)+\mathrm{N}_X(x(\theta)).
\end{align}
%Under Assumption \ref{ass:set}-i) and iii), it follows from Proposition 1.4.2 in \cite{facchinei2003finite} that for any given $\theta\in\R^n$, a solution $x(\theta)\in X$ to the variational inequality
%\begin{align}\label{optcon}
%G(x(\theta),\theta)^{\T}(x-x(\theta))\geq 0,\ \forall x\in X
%\end{align}
%is equivalent to an NE of game (\ref{pro:game})
The existence of the NE is guaranteed from Corollary 2.2.5 in \cite{facchinei2003finite}. Additionally, Assumption \ref{ass:pseudo-gradient}-i) ensures the uniqueness of the  NE by Theorem 2.3.3 in \cite{facchinei2003finite}.%, which is also made in \cite{Bianchi21,meng2022linear,Nguyen23,FengRR24,Carnevale2022Tracking-based}.

%{\color{blue}
%\begin{ass}\label{ass:second-order}
%$G(x,\theta)$ is strictly differentiable at $(x(\theta),\theta)$.
%\end{ass}}
\begin{rem}
Assumption \ref{ass:pseudo-gradient} holds in numerous real-world scenarios. For instance,  $f_i(x_i,x_{-i},\theta)$ takes the form of $g_i(x_i,x_{-i})+x_i^{\T}Q_i\theta$. Assume that the pseudo-gradient of $g_i(x_i,x_{-i})$ is $\mu$-strongly monotone and $l$-Lipschitz, which are rather common in existing literature including economic investment and smart grids fields \cite{Shakarami23,Maheshwari22,Maljkovic23}. %Further, if $g_i(x_i,x_{-i})$ is twice continously differentiable with respect to $x$, Assumption \ref{ass:pseudo-gradient}-i) is satisfied.
\end{rem}

The communication {\color{blue}graph among low-level players} is characterized by a directed graph $\mathcal{G}=(\mathcal{V},\mathcal{E},\mathcal{A})$, where $\mathcal{V}=[N]$ is the player set, $\mathcal{E}\subseteq \mathcal{V}\times\mathcal{V}$ is the directed edge set, and $\mathcal{A}=[a_{ij}]\in\R^{N\times N}$ is the adjacency matrix. $a_{ij}>0$ if $(j,i)\in\mathcal{E}$ or $i=j$, and $a_{ij}=0$, otherwise. Here, a directed edge $(j,i)\in\mathcal{E}$ means player $j$ could send information to player $i$. Let  $\mathcal{N}_{i}:=\{j\in[N]:(j,i)\in\mathcal{E}\}$ be the set of  neighbors of player $j$. A directed path from $i_1$ to $i_s$ is defined as a sequence of edges $(i_p,i_{p+1})\in\mathcal{E}, p\in[s-1]$ for different nodes. A directed graph is said to be strongly connected if there is a directed path from every node to any other node. It is assumed that $a_{ii}>0$ for all $i\in[N]$ in this paper. In the following, a standard condition on players' interaction is imposed.
\begin{ass}\label{ass:network-}
$\mathcal{G}$ is strongly connected. The adjacency matrix $\mathcal{A}$ is doubly stochastic, i.e., $\mathcal{A}\ve{1}_N=\mathcal{A}^{\T}\ve{1}_N=\ve{1}_N$.
\end{ass}

Assumption \ref{ass:network-} is a necessary condition for ensuring that players can estimate all other players' strategies. Under Assumption \ref{ass:network-}, there exists  $\bar{\sigma}\in (0,1)$ such that $\sigma_{N-1} (\mathcal{A})\leq \bar{\sigma}$.

Generally speaking, by virtue of the properties of objective functions and constraint sets of the players, various methods can be employed to compute the NE of the low-level game (\ref{pro:game}). Hence, our main focus shifts to addressing the high-level optimization problem in  (\ref{prob}). From the preceding analysis, it is apparent that to minimize the high-level function $f(x(\theta),\theta)$ with respect to $\theta$, where $x(\theta)$ is the NE of the low-level game (\ref{pro:game}),
one needs to firstly solve the low-level game problem (\ref{pro:game}), while its corresponding optimal decision $x(\theta)$ in turn depends on $\theta$. This mutual coupling between $x$ and $\theta$ renders problem (\ref{prob}) intrinsically hard to solve. Naturally, a direct and efficient approach involves  gradient-based methods, in which an implicit gradient of the high-level function $f(x(\theta),\theta)$ is required to be computed by utilizing the chain rule, as follows
\begin{align*}
\nabla f(x(\theta),\theta)&=\nabla x(\theta)\nabla_x f(x(\theta),\theta)+\nabla_\theta f(x(\theta),\theta).
\end{align*}
However, several challenges arise:
\begin{itemize}
%Unavailability\ of\ exact\ solutions\ of\ x(\theta).
\item [i)] For a particular $\theta\in\R^n$, obtaining a closed-form characterization of $x(\theta)$, where $x(\theta)$ is the NE of the low-level game (\ref{pro:game}), is often computationally expensive or even  unavailable unless the functions $f_i$, $i\in[N]$ take a very special form. Consequently, this typically limits the applicability of standard first-order and zeroth-order schemes.
\item [ii)] The computation of implicit gradient $\nabla f(\theta,x(\theta))$ requires not only access to $x(\theta)$ but also the assumption of differentiability of $x(\theta)$. Unfortunately, as demonstrated  by Example \ref{exam-nondiff}, $x(\theta)$ and hence $f(x(\theta),\theta)$ are non-differentiable in general, rendering the gradient-based methods infeasible. %Furthermore, one may not compute subgradients or Clarke generalized gradients easily in such settings either  \cite{Qiu2023}.
    As a result, the design of computational methods for addressing problem (\ref{prob}) remains  challenging.
\item [iii)] The function $f(x(\theta),\theta)$ may exhibit nonconvexity in $\theta$, as demonstrated by Example \ref{exam-nondiff}. This, in turn, poses another challenge in developing an effective algorithm with convergence guarantees, especially in establishing the convergence rate.
\end{itemize}

\subsection{Randomized smoothness}
%In this subsection, a gradient approximation approach is introduced to facilitate the development of the subsequent algorithm.

Consider a function $g: \R^n\to\R$. Define the uniform smoothing $\hat{g}$ of $g$ as
\begin{align}\label{def:est-fun}
\hat{g}(z):=\E_{\nu\in\mathbb{B}^n}[g(z+\xi\nu)],\ \forall z\in\R^n,
\end{align}
where $\xi>0$ is the smoothing parameter. Then a two-point sampling gradient estimator is proposed as \cite{yi2021online}
\begin{align}\label{def:onepoint-est}
\hat{\nabla} g(z)=\frac{n}{\xi}(g(z+\xi u)-g(z))u,
\end{align}
where $u\in\mathbb{S}^n$ is a stochastic vector with uniform distribution.

The following lemma shows some properties of $\hat{g}$ and $\hat{\nabla} g$.

\begin{lem}\cite{yi2021online}\label{lem:gradient-app}
~~~~~~~~~~~~~~~~~
\begin{itemize}
\item [i)] $\hat{g}(z)$ is differentiable even when $g$ is not.  Especially, it holds that
\begin{align}\label{lem:gradient-app-exp}
\nabla \hat{g}(z) = \E_{u \in \mathbb{S}^n}[\hat{\nabla}g(z)].
\end{align}
\item [ii)] If $g(z)$ is $L_g$-Lipschitz continuous ($L_g>0$), then $\nabla\hat{g}(z)$ is $\frac{n L_g}{\xi}$-Lipschitz continuous.
    %{\color{blue}i.e., for any $z,\ z'\in\R^n$,
%\begin{align}\label{222-0}
% \|\nabla\hat{g}(z)-\nabla\hat{g}(z')\|\leq \frac{n L_g}{\xi}\|z-z'\|.
%\end{align}}
Moreover, it holds that
\begin{align}\label{222}
\|\hat{\nabla} g(z)\|\leq nL_g.
\end{align}
%\item [ii)] If $g(z)$ is $L_g$-Lipschitz ($L_g>0$), then $\hat{g}(z)$ and $\nabla\hat{g}(z)$ are, respectively, $L_g$-Lipschitz and $\frac{n L_g}{\xi}$-Lipschitz. Moreover, it holds that
%\begin{align}
%&|\hat{g}(z)-g(z)| \leq \xi L_g,\label{111}\\
% &\|\hat{\nabla} g(z)\|\leq nL_g.\label{222}
%\end{align}
\end{itemize}
\end{lem}

\subsection{Moreau Smoothness}
Let $\chi_{\Theta}:\R^n\to\R\cup+\infty$ be an indicator function defined on a nonempty closed and convex set  $\Theta\subseteq\R^n$, that is,
\begin{align*}
&\chi_{\Theta}(\theta):=
\left\{
\begin{aligned}
  &0,  & \text{if}\ \theta\in\Theta, \\
  &+\infty,      & \text{if}\ \theta\notin\Theta.
\end{aligned}
\right.
\end{align*}
The Moreau smoothing  $\hat{\chi}_{\Theta}$ of $\chi_{\Theta}$ is defined as \cite{Beck2017}
\begin{align}\label{def:indismooth}
\hat{\chi}_{\Theta}(\theta):=\frac{1}{2\xi}{\dist}^2(\theta,\Theta),
\end{align}
where $\xi$ is the smoothing parameter.
By Theorem 6.60 in \cite{Beck2017}, $\hat{\chi}_{\Theta}$ is differentiable and
\begin{align*}
\nabla \hat{\chi}_{\Theta}(\theta) = \frac{1}{\xi}(\theta-\Pi_{\Theta}(\theta)).
\end{align*}
Moreover, $\nabla \hat{\chi}_{\Theta}(\theta)$ is $\frac{1}{\xi}$-Lipschitz continuous.

%\begin{rem}
%Intuitively, a gradient approximation approach is introduced to facilitate the development of the subsequent gradient-free method, which employs the function $\hat{f}$ as a substitute for the original function $f$, given that they are close enough when $\xi$ is chosen to be small from ii) of Lemma \ref{lem:gradient-app}. Additionally, $\nabla\hat{f}$ can be effectively approximated using the gradient estimator $\hat{\nabla}f$ by i) of Lemma \ref{lem:gradient-app}.
%\end{rem}
\begin{rem}
Intuitively, the randomized smoothed version $\hat{g}$ and Moreau smoothed variant $\hat{\chi}_{\Theta}$, as substitutes for original functions $g$ and $\chi_{\Theta}$, respectively, are typically introduced to contend with the nonsmoothness, which greatly help to develop a zeroth-order algorithm.
\end{rem}

\section{Main results}\label{sec3}
In this section, a zeroth-order algorithm for addressing the constrained, nonsmooth, and nonconvex problem (\ref{prob}) is developed, which allows for leveraging inexact solutions of the low-level game. Moreover, the convergence rate of the proposed algorithm is rigorously derived.

For notational convenience, define
\begin{align*}
   F_i(\theta):=f_i(x(\theta),\theta)),\ F(\theta):=\sum_{i=1}^{N}F_i(\theta),
\end{align*}
where $x(\theta)$ is the NE of the low-level game \eqref{pro:game}. To proceed, we first note that the constrained optimization problem (\ref{prob}) is  equivalent to
\begin{align}\label{pro:uncon}
\min_{\theta\in\R^n} \mathbf{F}(\theta):=F(\theta)+\chi_{\Theta}(\theta),
\end{align}
where %the indicator function
%$\chi_{\Theta}(\theta) := \begin{cases}
%	\displaystyle 0,       & \text{if } \theta \in \Theta \\
%	\displaystyle +\infty, & \text{if } \theta \notin \Theta
%\end{cases}$
\begin{align*}
	&\chi_{\Theta}(\theta):=
	\left\{
	\begin{aligned}
		&0,  & \text{if}\ \theta\in\Theta, \\
		&+\infty,      & \text{if}\ \theta\notin\Theta.
	\end{aligned}
	\right.
\end{align*}
Consider a smoothed approximation of (\ref{pro:uncon}) based on (\ref{def:est-fun}) and (\ref{def:indismooth}), given as
%\min_{\theta\in\Theta} F(\theta):=\sum_{i=1}^{N}F_i(\theta)
%\begin{align}\label{pro-smooth}
%\min_{\theta\in\Theta}\hat{F}(\theta):=\sum_{i=1}^N \hat{F}_i(\theta),
%\end{align}
\begin{align}\label{pro:opt-smooth}
\min_{\theta\in\R^n} \hat{\mathbf{F}}(\theta):=\hat{F}(\theta) + \hat{\chi}_{\Theta}(\theta).
\end{align}
where $\hat{F}(\theta):=\sum_{i=1}^{N}\hat{F}_i(\theta)$ with $\hat{F}_i(\theta):=\E_{\nu \in \mathbb{B}^n}[F_i(\theta+\xi \nu)]$ and $\hat{\chi}_{\Theta}(\theta):=\frac{1}{2\xi}{\dist}^2(\theta,\Theta)$ with $\xi>0$ being the smoothing parameter. Denote by $\hat{\mathbf{F}}^*:=\min_{\theta\in\R^n}\hat{\mathbf{F}}(\theta)$ the optimal value of problem (\ref{pro:opt-smooth}).

For nonsmooth and nonconvex optimization problem \eqref{pro:uncon}, it is shown in \cite{zhang2020complexity} that no algorithm could find an $\epsilon$-stationary point, i.e., a point $\theta$ for which $\min\{\|\zeta\|\mid \zeta\in\partial \mathbf{F}(\theta)\} \leq \epsilon$ in finite time.
%Specifically,
%\begin{align*}
%\min\{\|\zeta\|\mid \zeta\in\partial f(x)\} \leq \epsilon.
%\end{align*}
Hence, as a tractable and reasonable optimality criterion, a relaxed notion ($\delta, \epsilon$)-stationarity is considered for a vector $\theta$ satisfying $\min\{\|\zeta\|\mid \zeta \in \partial_\delta \mathbf{F}(\theta) \}\leq \epsilon$, where the set $\partial_\delta \mathbf{F}(\theta) := \operatorname{conv}\{\zeta\ |\ \zeta \in \partial \mathbf{F}(\theta'),\|\theta'-\theta\| \leq \delta\}$ denotes the $\delta$-Clarke generalized gradient of $\mathbf{F}$ at $\theta$ \cite{gold1977}.

Based on the analysis above, the goal of this paper is to develop a provably convergent algorithm to find a ($\delta, \epsilon$)-stationary point of function $\mathbf{F}$. % {\color{blue} and $\hat{\Theta}^*:=\{\theta\in\R^n: \hat{\mathbf{F}}(\theta)=\hat{\mathbf{F}}^*\}$ its optimal solution set.}
%For nonsmooth nonconvex objectives, $\epsilon$-stationarity cannot be guaranteed in finite time \cite{}. As a relaxation,  ($\delta, \epsilon$)-stationarity is a tractable criterion. To this end, the goal of this paper is to identify a ($\delta, \epsilon$)-stationary point of the global function $\mathbf{F}$.
%
% Toward this end, the $\delta$-Clarke generalized gradient of $\mathbf{F}$ at $\theta$ is introduced \cite{gold1977}, denoted by $\partial_\delta \mathbf{F}(\theta)$, as follows
%\begin{align}\label{def:clarke}
%\partial_\delta \mathbf{F}(\theta) := \operatorname{conv}\{\zeta\ |\ \zeta \in \partial \mathbf{F}(\theta'),\|\theta'-\theta\| \leq \delta\} .
%\end{align}
%Basically, the set $\partial_\delta \mathbf{F}(\theta)$ is shown to be nonempty, compact, and convex \cite{gold1977}. %Furthermore, for any $\xi>0$ and $\theta \in \mathbb{R}^n$, $\nabla \hat{\mathbf{F}}(\theta) \in \partial_{2 \xi} \mathbf{F}(\theta)$ \cite{mayne84}.
Before proceeding, we formalize the relationship between the original nonsmooth problem (\ref{prob}) and its smoothed counterpart (\ref{pro:opt-smooth}).

\begin{lem}\label{lem:clarkegra}
Consider problem (\ref{prob}) and let Assumptions \ref{ass:set} and \ref{ass:pseudo-gradient} be satisfied.
\begin{itemize}
\item [i)] If $\Theta=\R^n$, then for any $\xi>0$, it has $\nabla \hat{F}(\theta) \in \partial_{2 \xi} F(\theta)$.
%Consequently, there holds
%\begin{align}
%\ve{0} \in \nabla \hat{F}(\theta)+\mathcal{N}_{\Theta}(\theta) \Longrightarrow \ve{0}\in \partial_{2 \xi} F(\theta)+\mathcal{N}_{\Theta}(\theta).
%\end{align}
\item [ii)] For any $\delta>0$, if $\nabla\hat{\mathbf{F}}(\theta)=\ve{0}$ and $\xi\leq\frac{\delta}{\max\{2,nNL_F\}}$, then it holds that $\ve{0}\in\partial_{\delta} \mathbf{F}(\theta)$ and $\|\theta-\Pi_\Theta(\theta)\|\leq \xi nNL_F$.
\end{itemize}
\end{lem}
\begin{proof}
See Appendix \ref{proof-lem:clarkegra}.
\end{proof}

According to Lemma \ref{lem:clarkegra}, any stationary point of the $\xi$-smoothed problem satisfies an approximate stationary property for the original problem. Furthermore, the distance between the stationary point and the feasible set $\Theta$ is $\mathcal{O}(\xi)$. These findings lay the foundation for developing effective schemes to compute approximate stationary points of minimization of $\mathbf{F}$ in nonconvex and nonsmooth regimes.

Next, let us proceed to the algorithm design. In order to deal with the smoothed problem (\ref{pro:opt-smooth}), by virtue of Lemma \ref{lem:gradient-app}, for every $i\in[N]$, a two point sampling gradient approximation of $\hat{F}_i(\theta)$ is
\begin{align*}%\label{alg:exact-gra-i}
\hat{\nabla} F_i(\theta)=\frac{n}{\xi}(f_i(x(\theta+\xi u),\theta+\xi u)-f_i(x(\theta),\theta))u.
\end{align*}
Hence,
\begin{align}\label{alg:exact-gra}
\hat{\nabla} F(\theta)&= \sum_{i=1}^N\hat{\nabla} F_i(\theta)\notag\\
&=\sum_{i=1}^N\frac{n}{\xi}(f_i(x(\theta+\xi u),\theta+\xi u)-f_i(x(\theta),\theta))u.
\end{align}
Based on the preceding discussion, it is impractical to obtain the exact solution $x(\theta)$,  where $x(\theta)$ is the NE of the low-level game (\ref{pro:game}). We carefully address this challenge by introducing inexact evaluation of $x(\theta)$, where players are allowed to compute an $\varepsilon$-approximation solution $x_\varepsilon(\theta)$ satisfying
\begin{align}\label{defn:error}
\E[\|x_\varepsilon(\theta)-x(\theta)\|^2|\theta]\leq\varepsilon.
\end{align}
As a consequence, one needs to consider the inexact zeroth-order gradient, denoted by
\begin{align*}%\label{alg:inexact-gra-i}
\hat{\nabla}_{\varepsilon} F_i(\theta)=\frac{n}{\xi}(f_i(x_{\varepsilon}(\theta+\xi u),\theta+\xi u)-f_i(x_{\varepsilon}(\theta),\theta))u.
\end{align*}

\begin{algorithm}[H]
	\SetAlgoLined
	\SetAlgoNoLine
	\SetAlgoNlRelativeSize{-1}
	\SetNlSty{textbf}{}{}
	\caption{Inexact Smoothing-Enabled Zeroth-Order Method for Social Optimization}
	\label{alg:nonconvex}
	\KwIn{Stepsize $\alpha>0$, smoothing parameter $\xi>0$, and initial vector  $\theta_0 \in \Theta$.}
	\For{$k = 0,1,\ldots,K-1$}{
        Select vector $u_k\in\mathbb{S}^n$ independently and uniformly at random.\\
         Send $\theta_k+\xi u_k$ and $\theta_k$ to all players.\\
         Apply Algorithm \ref{alg:game} to receive inexact evaluations $x_{\varepsilon_k}(\theta_k+\xi u_k)$ and  $x_{\varepsilon_k}(\theta_k)$  of the corresponding NEs of the low-level game \eqref{pro:game}.\\
        Evaluate the inexact zeroth-order gradient approximation $\hat{\nabla}_{\varepsilon_k} F(\theta_k)$ by (\ref{alg:inexact-gra-}).\\
        Update $\theta_k$ as (\ref{alg:theta-}).\\
	}
\end{algorithm}

\begin{algorithm}[H]
	\SetAlgoLined
	\SetAlgoNoLine
	\SetAlgoNlRelativeSize{-1}
	\SetNlSty{textbf}{}{}
	\caption{ Distributed NE Seeking}
	\label{alg:game}
	\KwIn{Stepsize $\gamma>0$, and initial vector  $x_{i,0} \in X_i$, $ {\bf x}_{i,0}^{-i}\in\R^{d-d_i}$.}
	%\KwOut{A sequence of global functions:
		%	$$f(x; \{q_{t, i}, w_{t, i}\}:  i \in [M]), \quad t = 1, \ldots, T$$}
	\For{$t = 0,1,\ldots,t_k-1$}{
		\For{\rm $i = 1, ..., N$ in parallel}{
			Receive $\tilde{\theta}_k$ (that is, $\theta_k+\xi u_k$ or $\theta_k$) from the high-level regulator.\\
			Update
			\begin{subequations}
				\begin{align}%\label{alg:x}
					\hat{{\bf x}}_{i,t}  &= \sum_{j=1}^Na_{ij}{\bf x}_{j,t},\label{alg:x1}\\
					x_{i,t+1} &= \Pi_{X_i}[\hat{{\bf x}}_{i,t}^i-\gamma\nabla_i f_i(\hat{{\bf x}}_{i,t},\tilde{\theta}_k)],\label{alg:x2}\\
					{\bf x}_{i,t+1}^{-i} &= \hat{{\bf x}}_{i,t}^{-i}.\label{alg:x3}
				\end{align}
			\end{subequations}
		}
	}
	Players send $x_{i,t_k}, i\in[N]$  to the high-level regulator.
\end{algorithm}

In a nutshell, an inexact smoothing-enabled zeroth-order method is developed in Algorithm \ref{alg:nonconvex}. At each iteration $k$, firstly, the high-level regulator selects vector $u_k\in\mathbb{S}^n$ independently and  uniformly at random and then sends $\theta_k+\xi u_k$ and $\theta_k$ to all players. Simultaneously, the high-level regulator makes an inexact call to a  distributed NE seeking algorithm, i.e., Algorithm \ref{alg:game}, inspired by \cite{Bianchi21} twice to compute inexact approximations of NEs $x(\theta_k+\xi u_k)$ and $x(\theta_k)$ of the low-level game \eqref{pro:game}, denoted by $x_{\varepsilon_k}(\theta_k+\xi u_k)$ and $x_{\varepsilon_k}(\theta_k)$, respectively.  Specifically,  at time $t$, each player $i$ maintains variables $x_{i,t}$ and ${\bf x}_{i,t}^j$ to represent its own strategy and estimate of player $j$'s strategy, respectively. Define ${\bf x}_{i,t}^i := x_{i,t}$ and also ${\bf x}_{j,t}^{-i}:=\col\{{\bf x}_{j,t}^l\}_{l\in[N]\setminus i}$. Let ${\bf x}_{i,t}:=\col\{{\bf x}_{i,t}^j\}_{j\in[N]}\in\R^{d}$ and ${\bf x}_t:=\col\{{\bf x}_i\}_{i\in[N]}\in\R^{Nd}$.
In the partial-decision information setting, players update their estimates based on the information received from their neighbors by applying dynamic consensus technique (\ref{alg:x1}), and subsequently adjust their own decisions through a gradient descent step (\ref{alg:x2}). Herein, it is mentioned that the partial gradient $\nabla_i f_i$ is evaluated on the local estimate $\hat{{\bf x}}_{i,t}$, while not employing the actual decision $x_{i,t}$.
After receiving the estimated feedbacks $x_{\varepsilon_k}(\tilde{\theta_k}):=\col\{x_{i,t_k}\}_{i\in[N]}$  from the players, where  $\tilde{\theta}_k\in\{\theta_k, \theta_k+\xi u_k \}$,  the inexact zeroth-order gradient $\hat{\nabla}_{\varepsilon_k} F(\theta_k)$ is computed as
\begin{align}\label{alg:inexact-gra-}
\hat{\nabla}_{\varepsilon_k} F(\theta_k)&= \sum_{i=1}^N\hat{\nabla}_{\varepsilon_k} F_i(\theta_k)\notag\\
&=\sum_{i=1}^N\frac{n}{\xi}(f_i(x_{\varepsilon_k}(\theta_k+\xi u_k),\theta_k+\xi u_k)\notag\\
&\quad-f_i(x_{\varepsilon_k}(\theta_k),\theta_k))u_k.
\end{align}
At last, the high-level regulator updates its decision using gradient descent
\begin{align}\label{alg:theta-}
\theta_{k+1}=\theta_k-\alpha (\hat{\nabla}_{\varepsilon_k} F(\theta_k)+\frac{1}{\xi}(\theta_k-\Pi_\Theta(\theta_k))),
\end{align}
where $\alpha>0$ is the step size to be determined.
%To cope with partial-decision information, each agent keeps an estimate of all other agents’ actions. Let ${\bf x}_i=\col\{{\bf x}_i^j\}_{j\in[N]}$, where ${\bf x}_i^i := x_i$ and ${\bf x}_i^j$ is agent $i$'s estimate of agent $j$'s action, for all $j\neq i$; also, ${\bf x}_j^{-i}=\col\{{\bf x}_j^l\}_{l\in[N]\setminus i}$. In the partial-decision information setting, the partial gradient $\nabla_i f_i$  are evaluated on the local estimates ${\bf x}_{i,t}^{-i}$, not on the actual strategies $x_{-i,t}$. The agents aim at asymptotically reconstructing the true value of the opponents’ actions, based on the data received from their neighbors. Each agent updates its estimates according to consensus dynamics, then its action via a gradient step. We remark that each agent computes the partial gradient of its cost in its local estimates ${\bf x}_i$, not on the actual joint action ${\bf x}$.%players depart from the conventional approach of evaluating gradient at actual decisions $\nabla_i f_i(x_t,\theta_k+\xi_ku_k)$, and instead compute the gradient based on local estimates contingent on $\theta_k$, that is $\nabla_i f_i({\bf x}_{i,t},\theta_k+\xi_ku_k)$.

\begin{rem}
Utilizing inexact estimation of the NE of the low-level game inherently introduces a bias in the approximation of the zeroth-order gradient.  This bias, if uncontrolled, has the potential to undesirably propagate and impact the overall performance of  the algorithm. Notably, the degree of inexactness in $x_{\varepsilon_k}(\tilde{\theta}_k)$ plays a pivotal role in the convergence analysis, which is precisely controlled by the termination criterion $t_k$ specified in Algorithm \ref{alg:game}. This criterion will be formally derived in the subsequent analysis presented in Theorem \ref{thm-main-nonconvex} to establish the convergence rate.
\end{rem}

\begin{rem}
It is worthwhile noting that the regulator in the high-level does not need to acquire all information on the players. In fact, at moment $k+1$, only two function values $f_i(x_{\varepsilon_k}(\theta_k+\xi u_k),\theta_k+\xi u_k)$ and $f_i(x_{\varepsilon_k}(\theta_k),\theta_k)$ are required to be disclosed by the high-level regulator to update its decision $\theta_k$. Moreover, compared to gradient-based algorithms \cite{Maljkovic2024OnDC,linon24,jo2023}, the regulator does not need to compute players' gradient information, making this more aligned with real-world scenarios. In addition, it is noteworthy that in the setting of the  partial-decision information regarding the underlying game, each individual strategy and gradient of the  objective function are not accessible to other players. Consequently, every player must exert extra effort to estimate strategies of other players and hence steps (\ref{alg:x1}) and (\ref{alg:x3}) are introduced in Algorithm \ref{alg:game}. However, during the practical implementation of Algorithm \ref{alg:nonconvex}, it is only necessary to employ simple gradient descent method to update the players' decision variables to obtain approximate solution $x_{\varepsilon_k}(\tilde{\theta}_k)$ of the low-level game, which will simplify Algorithm \ref{alg:nonconvex} greatly.
\end{rem}

%\begin{rem}
%{\color{blue}Most existing bilevel optimization methods require the differentiability of the high-level objective function, which frequently fail to hold in constrained scenarios and thus are not applicable to our problem. Besides, these algorithms are often developed for unconstrained problems or scenarios involving function constraints in the low-level (see \cite{Ankur18,ghadimi2018approximation,pmlr-v202-khanduri23a} and references therein). Note that bilevel optimization with low-level set constraints is relatively few investigated and remains challenging. }%\cite{cui2023}
%\end{rem}

Given $\theta\in\R^n$, define the extended pseudo-gradient mapping $\mathbf{G}:\R^{Nd}\times\Theta\to\R^d$ as
\begin{align*}
\mathbf{G}({\bf x},\theta):=\col\{\nabla_if_i(x_i,{\bf x}_i^{-i},\theta)\}_{i\in[N]}.
\end{align*}
It follows from Lemma 1 in  \cite{Bianchi21} that  $\mathbf{G}({\bf x},\theta)$ is Lipschitz continuous with respect to ${\bf x}$,  as stated formally below.
\begin{lem}
Under Assumptions \ref{ass:set}-iii) and \ref{ass:pseudo-gradient}-i), for any $\theta\in\R^n$, the mapping $\mathbf{ G}$  is $l'$-Lipschitz continuous for some $\mu\leq l'\leq l$, that is,
\begin{align*}
\mathbf{G}({\bf x},\theta)-\mathbf{G}({\bf x}',\theta)\leq l'\|{\bf x}-{\bf x}'\|,\ \forall {\bf x}, {\bf x}'\in \R^d.
\end{align*}
\end{lem}

To proceed further analysis, the convergence result of  Algorithm \ref{alg:game} for resolving the low-level game (\ref{pro:game}) is presented.
\begin{lem}\label{thm:game}
Suppose that Assumptions \ref{ass:set}-i), iii), \ref{ass:pseudo-gradient}-i), and \ref{ass:network-} hold. Let $\{x_t\}_{t=0}^{t_k}$ be the sequence generated by Algorithm \ref{alg:game} at the $k$-th iteration of Algorithm \ref{alg:nonconvex}. If the stepsize $\gamma>0$ is chosen such that
\begin{align}\label{game_step}
\gamma&<\min\{1,\frac{\bar{\sigma}}{3l},\frac{2\mu}{l^2},\frac{2\mu(1-\bar{\sigma}^2)}{a}\},
\end{align}
with $a:=\bar{\sigma}^2(2ll'+l'^2+4\mu l'+2l^2)+2(l^2l'^2+\mu l'^2+2l^2l'^2)\bar{\sigma}^2+2l^2l'^2\bar{\sigma}^2$, then
\begin{align}\label{thm-gamecon}
\|{\bf x}_{t_k}-{\bf 1}_N\otimes x(\tilde{\theta}_k)\|^2\leq  q^{t_k} \|{\bf x}_0-{\bf 1}_N\otimes x(\tilde{\theta}_k)\|^2,
\end{align}
where $q:=\|Q_\gamma\|\in(0,1)$ with
\begin{equation*}
Q_\gamma=
\left[
\begin{array}{ccc}
    1-\frac{2\gamma\mu}{N}+\frac{\gamma^2l^2}{N} & \frac{\gamma(l+l')+\gamma^2ll'}{\sqrt{N}}\bar{\sigma}\\
    \frac{\gamma(l+l')+\gamma^2ll'}{\sqrt{N}}\bar{\sigma} & (1+2\gamma l+\gamma^2l^2)\bar{\sigma}^2
\end{array}
\right].
\end{equation*}
\end{lem}
\begin{proof}
See Appendix \ref{proof-thm:game}.
\end{proof}
%\begin{proof}
%Note that $x(\tilde{\theta}_k)$ is the NE of the low-level game (\ref{pro:game}) under  $\theta=\tilde{\theta}_k$. Then the proof can be derived following the same line in Theorem 1 and Lemma 2 in \cite{Bianchi21} together with  the definition of the pseudo-gradient mapping $G(x,\theta).$
%\end{proof}

It follows from (\ref{thm-gamecon}) in Lemma \ref{thm:game} %and the definition of $\varepsilon_k$ in (\ref{defn:error}),
that
\begin{align}\label{defn_error_k}
\varepsilon_k&:=\|x_{t_k}-x(\tilde{\theta}_k)\|^2\notag\\
&\leq q^{t_k}\|{\bf x}_0-{\bf 1}_N\otimes x(\tilde{\theta}_k)\|^2\notag\\
&\leq 2(\|{\bf x}_0\|^2+NB_X^2) q^{t_k},
\end{align}
where the last inequality is derived owing to (\ref{F-lips-1}).
\begin{rem}
As is shown in (\ref{defn_error_k}), there exists a tradeoff between the inexactness of Algorithm \ref{alg:game} and the iteration number $t_k$, which will be carefully specified as discussed later. Besides, different from Theorem 1 in \cite{Bianchi21}, an explicit upper bound on the require stepsize $\gamma$ for eatablishing linear convergence of Algorithm \ref{alg:game} is given here in (\ref{game_step}).
\end{rem}

The following lemma reveals the Lipschitz continuous of the upper-level function $f_i$, which is rather crucial in analyzing the algorithm.
\begin{lem}\label{lem:conti}
	Under Assumptions \ref{ass:set} and \ref{ass:pseudo-gradient}, $f_i(x(\theta),\theta)$ is Lipschitz continuous, i.e., there exists a constant $L_F>0$ such that
	\begin{align}\label{eq:imlip}
		&\|f_i(x(\theta),\theta)-f_i(x(\theta'),\theta')\|\leq L_F\|\theta-\theta'\|, \forall \theta\in\R^n.%\label{F-lips-3}
	\end{align}
\end{lem}
\begin{proof}
	See Appendix \ref{proof:lem:conti}.
\end{proof}

Define a natural filtration
\begin{align*}
&\F_0:=\{\emptyset,\Omega\},\notag\\
&\F_k:=\sigma\{\theta_0,u_t:0\leq t\leq k-1\},\ \forall k\geq 1.
\end{align*}
It can be seen that $\theta_k$ generated by Algorithm \ref{alg:nonconvex} is adapted to $\F_k$ and $u_k$ is independent of $\F_k$. Then, the following lemma on properties of the zeroth-order gradient is derived.

\begin{lem}\label{lem-inex-zero-grad}
Under Assumptions \ref{ass:set} and \ref{ass:pseudo-gradient}, there hold
\begin{align}
&\E[\hat{\nabla}F(\theta_k)+\nabla\hat{\chi}_\Theta(\theta_k)|\F_k]= \nabla \hat{\mathbf{F}}(\theta_k),\label{lem-inex-3}\\
%\E[\|\hat{\nabla}F(\theta_k)+\nabla\hat{\chi}_\Theta(\theta_k)\|^2|\F_k]&\leq 2\frac{n^2}{\xi^2}L_F^2+2\|\nabla\hat{\chi}_\Theta(\theta_k)\|^2,\label{lem-inex-1}\\
&\E[\|\hat{\nabla}_{\varepsilon_k}F(\theta_k)-\hat{\nabla}F(\theta_k)\|^2|\F_k]\leq 4N^2\frac{L_x^2n^2}{\xi^2}\varepsilon_k,\label{lem-inex-4}\\
&\E[\|\hat{\nabla}F(\theta_k)+\nabla\hat{\chi}_\Theta(\theta_k)-\nabla \hat{\mathbf{F}}(\theta_k)\|^2|\F_k]\leq N^2 n^2L_F^2.\label{lem-inex-2}
\end{align}
\end{lem}
\begin{proof}
See Appendix \ref{proof-lem-inex-zero-grad}.
\end{proof}

Equipped with the previous preparations, it is ready to present the main result of this paper now.

\begin{thm}\label{thm-main-nonconvex}
Suppose that Assumptions \ref{ass:set}-\ref{ass:network-} hold. Let $\{\theta_k\}$ be the sequence generated by Algorithm \ref{alg:nonconvex}, if $\alpha\leq\frac{\xi}{4(nNL_F+1)}$, then it holds
\begin{align}\label{thm-nonconvex}
&\frac{1}{K}\sum_{k=0}^{K-1}\E[\| \nabla\hat{\mathbf{F}}(\theta_k)\|^2]\notag\\
&\leq\frac{4(\hat{\mathbf{F}}(\theta_0)-\hat{\mathbf{F}}^*)}{\alpha K} + 8N^2n^2L_F^2\alpha\frac{nNL_F+1}{\xi}\notag\\
&\quad + \frac{20 N^2L_x^2n^2\sum_{k=0}^{K-1}\varepsilon_k}{\xi^2K}.
\end{align}
\end{thm}
\begin{proof}
	See Appendix \ref{proof-thm-main-nonconvex}.
\end{proof}

\begin{cor}\label{corollary-1}
Under the same conditions as in Theorem \ref{thm-main-nonconvex}, when $t_k$ in Algorithm \ref{alg:game} is chosen such that $t_k\geq\lceil \frac{-s\ln(k+1)}{\ln q}\rceil$ with $s\in[0,1)$ and $\alpha=\frac{\alpha_0}{\sqrt{K}}$ with $\alpha_0>0$, %and $\alpha=\frac{1}{\sqrt{K}}$,
 there holds
\begin{align}\label{cor:1}
\frac{1}{K}\sum_{k=0}^{K-1}\E[\| \nabla\hat{\mathbf{F}}(\theta_k)\|^2]=\mathcal{O}(\frac{1}{\sqrt{K}}+\frac{1}{K^s}).
\end{align}

\end{cor}
\begin{proof}
See Appendix \ref{proof-corollary-1}.
\end{proof}

\begin{cor}\label{cor-2-}
Under the same conditions as in Theorem \ref{thm-main-nonconvex}, when Algorithm \ref{alg:nonconvex} is executed in an exact way, that is, $\hat{\nabla}_{\varepsilon_k} F(\theta_k)$ is replaced by $\hat{\nabla} F(\theta_k)$ defined in (\ref{alg:exact-gra}) in the fifth step, and $\alpha\leq\frac{\xi}{2(nNL_F+1)}$, the following statement holds
\begin{align}\label{cor:2}
&\frac{1}{K}\sum_{k=0}^{K-1}\E[\| \nabla\hat{\mathbf{F}}(\theta_k)\|^2]\notag\\
&\leq\frac{2(\hat{\mathbf{F}}(\theta_0)-\hat{\mathbf{F}}^*)}{\alpha K} + 2N^2n^2L_F^2\frac{nNL_F+1}{\xi }\alpha.
\end{align}
\end{cor}

\begin{proof}
See Appendix \ref{proof-cor-2-}.
\end{proof}

\begin{cor}\label{corollary-3}
Under the same conditions as in corollary \ref{cor-2-}, let $\alpha=\frac{\alpha_0}{\sqrt{K}}$ with $\alpha_0>0$, %and $\alpha=\frac{1}{\sqrt{K}}$,
then it holds
\begin{align}\label{cor:3}
\frac{1}{K}\sum_{k=0}^{K-1}\E[\| \nabla\hat{\mathbf{F}}(\theta_k)\|^2]=\mathcal{O}(\frac{1}{\sqrt{K}}).
\end{align}

\end{cor}

%\begin{cor}\label{cor:exact}
%Under the same conditions as in Theorem \ref{thm-main-nonconvex}, let $\alpha=\frac{1}{\sqrt{K}}$, then it holds
%\begin{align}\label{cor:1}
%\frac{1}{K}\sum_{k=0}^{K-1}\E[\| \nabla\hat{\mathbf{F}}(\theta_k)\|^2]=\mathcal{O}(\frac{1}{\sqrt{K}}).
%\end{align}
%\end{cor}

\begin{rem}
%Notably, nonconvexity and nonsmoothness lead to the NP-hard complexity of the studied problem, as the existence of multiple optima, the lack of structure properties to guarantee the presence of optima, such as convexity and strong convexity, we cannot evaluate the performance of the algorithm by optimal value or iteration sequence. Hence, as done is most existing nonconvex works, we choose gradient. Due to the nonsmoothness, we choose appromiation gradient.
Notably, the relations shown in Lemma \ref{lem:clarkegra} together with the derived convergence results above imply that Algorithm \ref{alg:nonconvex} could find a ($\delta, \epsilon$)-stationary point of $\mathbf{F}$.
It can be found from Theorem \ref{thm-main-nonconvex} that there exists a tradeoff between the convergence bound and approximation accuracy, where a larger $\xi$ leads to a tighter bound, while a smaller $\xi$ reduces the approximation error relative to the original problem.
From Corollary \ref{corollary-1}, it is evident that in the inexact regime, with an appropriately selected inner iteration number $t_k$, our proposed algorithm can attain a  sublinear convergence rate of $\mathcal{O}(\frac{1}{\sqrt{K}}+\frac{1}{K^s})$ with $s\in[0,1)$. On the other hand, Lemma  \ref{thm:game} indicates that an increasing number of iteration $t_k$ derives a more accurate inner loop solution, which ultimately results in a faster convergence speed owing to a larger $s$. Specially, the order of magnitude of (\ref{cor:1}) stands at $\mathcal{O}(\frac{1}{\sqrt{K}})$ when choosing $t_k\geq\lceil \frac{-s\ln(k+1)}{\ln q}\rceil$ and $s\in[\frac{1}{2},1)$.
Moreover, according to Corollary \ref{corollary-3}, when the low-level game can be solved exactly, i.e., $\varepsilon_k=0$,\ $\forall k=0,\ldots, K-1$, the convergence rate is also $\mathcal{O}(\frac{1}{\sqrt{K}})$. Note that the bound $\mathcal{O}(\frac{1}{\sqrt{K}})$ recovers the existing  convergence rate of nonconvex nonsmooth optimization in \cite{lin2022gradientfree}.
Additionally, it should be emphasized that by utilizing the continuity of $x(\theta)$, if $\theta_k$ converges to the optimal solution $\theta^*\in\hat{\Theta}^*$, then $x(\theta)$ will correspondingly converges to $x(\theta^*)$.
\end{rem}

\begin{rem}
It is worth mentioning that the studies \cite{Shakarami23,wang22,liu22} investigated related problems concerning the social optimum, however, the problem settings are different. Besides, \cite{wang22} did not consider strategy set constraints or intervention constraints. In \cite{Shakarami23}, a kind of special quadratic games was considered, which possessed good properties on cost functions. Furthermore, both \cite{Shakarami23} and \cite{wang22} did not provide convergence rates of designed algorithms. Although the sublinear convergence rate was established in \cite{liu22} for a bilevel incentive design algorithm, this only applies to probability simplex constrained games and requires calculating the exact gradient of $x(\theta)$.
\end{rem}

%In particular, when choosing $t_k\geq\lceil \frac{-\ln(k+1)}{4\ln q}\rceil$, the convergence rate stands at $\mathcal{O}(\frac{1}{\sqrt{K}})$.

\section{Numerical simulation}\label{sec4}

%Consider the problem of Cournot competition inspired by \cite{Shakarami23}. Specially, $N$ companies produce different goods with market capacities constraints under the supervision of a central authority who regulates the whole markets through adjustable price policies $\theta$ varying within a specified range $\Theta=[\theta_{\min }, \theta_{\max }]$.  Every company $i$ produces $x_i\in [0, x_i^{\max}]$ quantity of good and aims to minimize its cost, given by $f_i(x_i, x_{-i},\theta):=\frac{1}{2}x_i^2 - x_i(az_i(x)+b_i) - x_i\theta$. Herein, $a\in \R$ captures the impact of network aggregate $z_i(x):=\sum_{j\in[N]}p_{ij}x_j$ on the individual cost of companies with $p_{ij}$ being the impact of $j$'s action on the cost of company $i$, $b_i\in\R$ is the standalone marginal return, and $x_i\theta$  represents the intervention of the high-level regulator in modifying the standalone marginal return $b_i$ to $b_i + \theta$. In order to maintain overall functioning and balance of the market, the central authority interacts with companies through pricing policies and is interested in miniming the overall cost of all companies.

Consider the problem of flexible electric vehicle charging control inspired by \cite{liu2023approximate}.
Specifically, there are $N$ residents in a neighborhood who commute daily using electric vehicles and need to charge their vehicles at a public charging station. The charging process is regulated by a central authority (such as a power company or government) through adjustable price policies $\theta$ within a specified range $\Theta=[\theta_{\min }, \theta_{\max }]$. Due to battery capacity limitations, every resident $i$ has its charging consumption being $x_i\in X_i= [0, x_i^{\max}]$ and aims to minimize its electricity bill, given by $f_i(x_i, x_{-i},\theta):= c_i(x_i-d_i)^2 +  a_ix_i + \lambda_i(x_i-\sum_{i=1}^Nx_i/N)^2 + r x_i\theta$. Herein, $c_i, d_i, a_i, r>0$ are constants, and $\lambda_i$ captures resident $i$'s sensitivity to deviation from average consumption. Moreover, $x_i\theta$  represents the charging cost imposed by the high-level regulator when resident $i$ selects a specific $x_i$. In order to maintain a balanced demand on the power grid demand and reduce total consumption of social resources,  the central authority interacts with residents and is interested in miniming the overall cost of all residents.

It is assumed that there are $10$ residents here. By linking any two distinct players with a probability of $\frac{1}{3}$, we randomly generate a network, whose adjacency matrices are constructed follwing the Metropolis rule. For every $i\in[10]$, set $c_i=4+i$, $d_i=7+2i$, $a_i=10+i$, $\lambda_i=0.1$, and $r=1$. Moreover, let set constraints $X_i=[0,25]^{10}$ and $\Theta=[1, 3]$.
Choose stepsizes $\alpha=0.00001$ and $\gamma=0.01$, smoothing parameter $\xi=0.0001$, and inner iteration number $t_k=\lceil5\ln(k+1)\rceil$. The evolution of residual $|\theta_k-\theta^*|$ versus iteration $k$ is plotted in Fig. \ref{fig:1}, demonstrating that the regulator can find the optimal decision successfully. Furthermore,
Fig. \ref{fig:2} gives rises to the trajectory of $| \nabla\hat{\mathbf{F}}(\theta_k)|$, from which it can be seen that the proposed algorithm converges fast. In addition, changes in each player $i$'s  objective value $f_i(x_{\varepsilon_k}(\theta_k),\theta_k)$, $i\in[10]$ and the regulator's objective value $\sum_{i=1}^{10}f_i(x_{\varepsilon_k}(\theta_k),\theta_k)$ are shown in Fig. \ref{fig:3}, respectively, which provide convergence results for seeking NE and social optimum.

\begin{figure}[H]
\centering
\includegraphics[width=8cm,height=5.5cm]{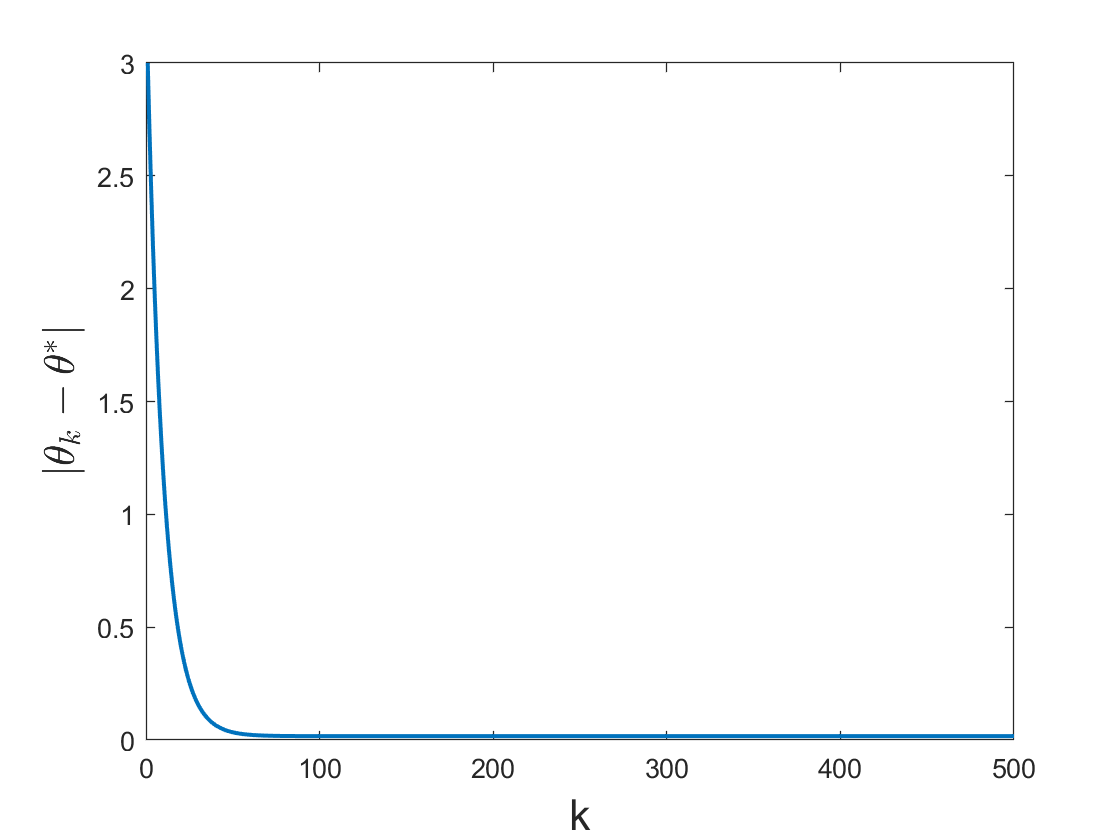}
\caption{The trajectory of $|\theta_k-\theta^*|$.\label{fig:1}}
\end{figure}

\vspace{-0.5cm}
%\begin{figure}[htbp]
%\centering
%\includegraphics[width=8cm,height=5.5cm]{3.png}
%\caption{The trajectory of $x_{\varepsilon_k}(\theta_k)$.\label{fig:3}}
%\end{figure}

\begin{figure}[H]
\centering
\includegraphics[width=8cm,height=5.5cm]{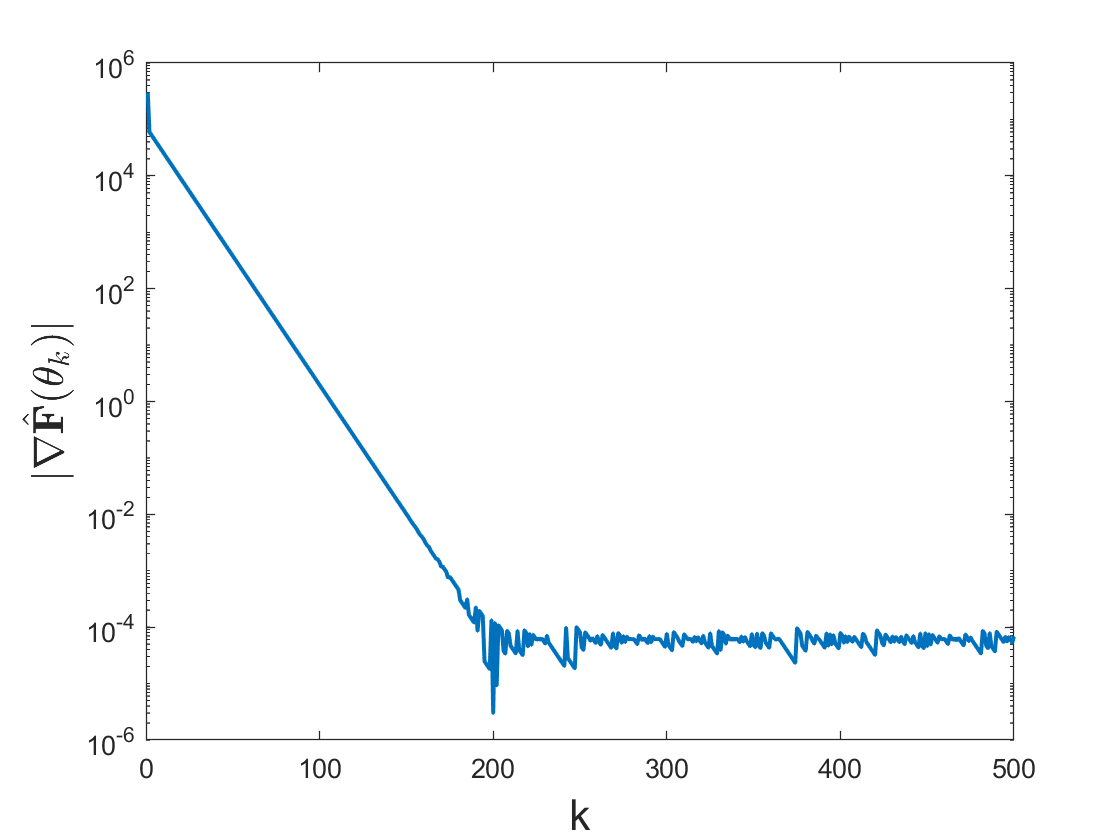}
\caption{The trajectory of $|\nabla\hat{\mathbf{F}}(\theta_k)|$.\label{fig:2}}
\end{figure}

\vspace{-0.5cm}

\begin{figure}[H]
\centering
\includegraphics[width=8cm,height=5.5cm]{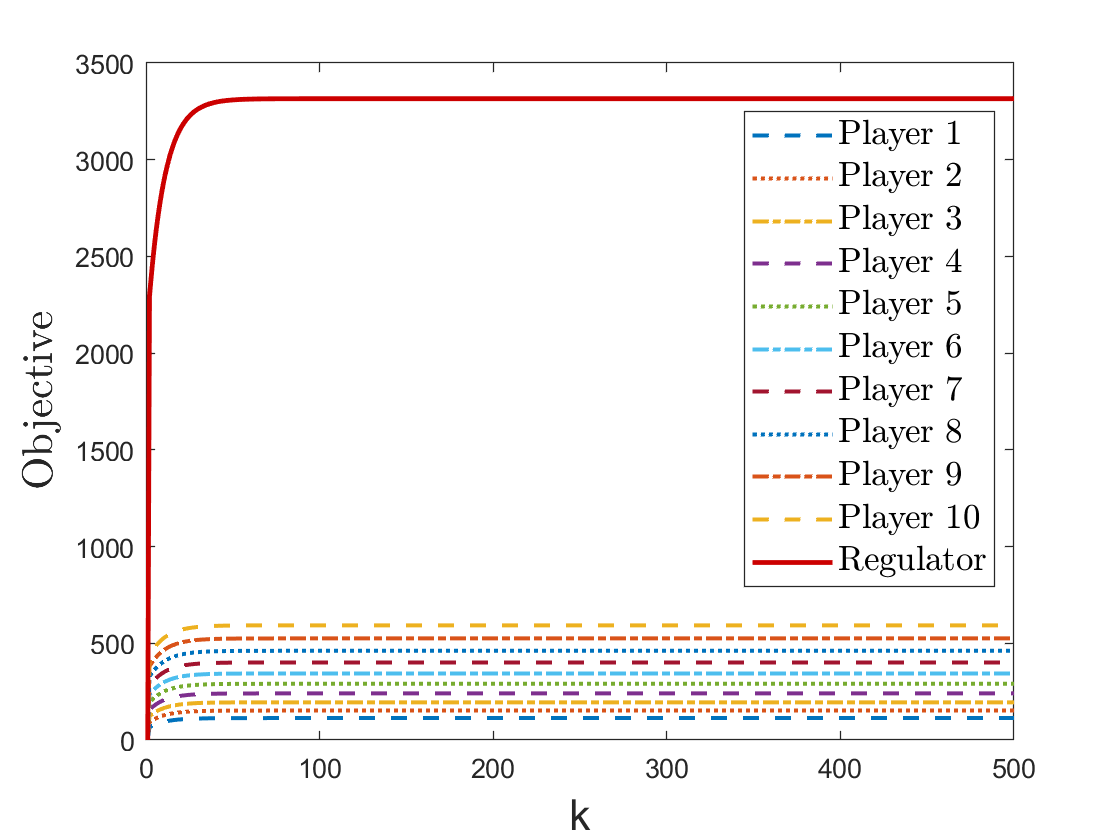}
\caption{Trajectories of players' objective values $f_i(x_{\varepsilon_k}(\theta_k),\theta_k)$, $i\in[10]$ and the regulator's objective value $\sum_{i=1}^{10}f_i(x_{\varepsilon_k}(\theta_k),\theta_k)$.\label{fig:3}}
\end{figure}

\section{Conclusion}\label{sec5}
In this work, a new framework inspired by noncooperative games, i.e., social optimization, has been explored, in which each cost function depends not only on all players' strategies, but also on an intervention decision of a high-level regulator. This problem possesses a bilevel structure, where the  low-level players aim to seek the NE associated with the regulator's decision in a noncooperative game, and the goal of the high-level regulator is to minimize the sum of all players' objectives to achieve the social optimum. To address this constrained, nonconvex and nonsmooth problem, a gradient-free method was developed by virtue of the randomized smoothness and Moreau smoothness techniques. It was proved that the designed algorithm converges at a sublinear rate, irrespective of whether the underly
ing game is solved exactly or inexactly. Potential avenues for future research may be on stochastic objective functions and incorporate equality and inequality constraints.

According to Lemma \ref{lem:clarkegra}, any stationary point of the $\xi$-smoothed problem satisfies an approximate stationary property for the original problem. Furthermore, the distance between the stationary point and the feasible set $\Theta$ is $\mathcal{O}(\xi)$. These findings lay the foundation for developing effective schemes to compute approximate stationary points of minimization of $\mathbf{F}$ in nonconvex and nonsmooth regimes.

Next, let us proceed to the algorithm design. In order to deal with the smoothed problem (\ref{pro:opt-smooth}), by virtue of Lemma \ref{lem:gradient-app}, for every $i\in[N]$, a two point sampling gradient approximation of $\hat{F}_i(\theta)$ is
\begin{align*}%\label{alg:exact-gra-i}
\hat{\nabla} F_i(\theta)=\frac{n}{\xi}(f_i(x(\theta+\xi u),\theta+\xi u)-f_i(x(\theta),\theta))u.
\end{align*}
Hence,
\begin{align}\label{alg:exact-gra}
\hat{\nabla} F(\theta)&= \sum_{i=1}^N\hat{\nabla} F_i(\theta)\notag\\
&=\sum_{i=1}^N\frac{n}{\xi}(f_i(x(\theta+\xi u),\theta+\xi u)-f_i(x(\theta),\theta))u.
\end{align}
Based on the preceding discussion, it is impractical to obtain the exact solution $x(\theta)$,  where $x(\theta)$ is the NE of the low-level game (\ref{pro:game}). We carefully address this challenge by introducing inexact evaluation of $x(\theta)$, where players are allowed to compute an $\varepsilon$-approximation solution $x_\varepsilon(\theta)$ satisfying
\begin{align}\label{defn:error}
\E[\|x_\varepsilon(\theta)-x(\theta)\|^2|\theta]\leq\varepsilon.
\end{align}
As a consequence, one needs to consider the inexact zeroth-order gradient, denoted by
\begin{align*}%\label{alg:inexact-gra-i}
\hat{\nabla}_{\varepsilon} F_i(\theta)=\frac{n}{\xi}(f_i(x_{\varepsilon}(\theta+\xi u),\theta+\xi u)-f_i(x_{\varepsilon}(\theta),\theta))u.
\end{align*}

\appendix

\section{Proof of Lemma \ref{lem:clarkegra}}\label{proof-lem:clarkegra}
\begin{proof}
The proof of conclusion i) follows from Proposition 2 in \cite{may84}. Conclusion  ii) is inspired by Proposition 2 in \cite{Qiu2023}.  From $\nabla\hat{\mathbf{F}}(\theta)=\ve{0}$, one can obtain
$\nabla\hat{F}(\theta)+\frac{1}{\xi}(\theta-\Pi_{\Theta}[\theta])=\ve{0}$.
%which implies that $\|\theta-\Pi_{\Theta}[\theta]\|=\xi\|\sum_{i=1}^{N}\nabla\hat{F}_i(\theta)\|$.
Note that
\begin{align*}
\|\nabla\hat{F}(\theta)\|&\leq \sum_{i=1}^{N}\|\nabla\hat{F}_i(\theta)\|\notag\\
&=\frac{n}{\xi}\sum_{i=1}^{N}\|\E_{u \in \mathbb{S}^n}[(F_i(\theta+\xi u)-F_i(\theta))u]\|\notag\\
&\leq \frac{n}{\xi}\sum_{i=1}^{N}\E_{u \in \mathbb{S}^n}[|F_i(\theta+\xi u)-F_i(\theta)|]\notag\\
&\leq nL_F\sum_{i=1}^{N}\E_{u \in \mathbb{S}^n}[\|u\|]\notag\\
&=nNL_F,
\end{align*}
where Lemma \ref{lem:gradient-app}-i) is utilized to obtain the first equality, the second inequality is derived by applying the Jensen's inequality, and the third inequality holds due to Lemma \ref{lem:conti}. Then  this combined with the preceding relation derives $\|\theta-\Pi_\Theta(\theta)\|\leq \xi nNL_F$.
Invoking $\nabla \hat{F}(\theta) \in \partial_{2 \xi} F(\theta)$ and $2\xi\leq\delta$, one has   $\nabla\hat{F}(\theta)\in\partial_\delta F(\theta)$. Hence, in order to get $\ve{0}\in\partial_{\delta} \mathbf{F}(\theta)$, noting $\nabla\hat{F}(\theta)+\frac{1}{\xi}(\theta-\Pi_{\Theta}(\theta))=\ve{0}$, it suffices to verify $\frac{1}{\xi}(\theta-\Pi_{\Theta}(\theta))\in \partial_{\delta}\chi_\Theta(\theta)$. Recall the definition of $\delta$-Clarke generalized gradient of $\chi_{\Theta}$ at $\theta$, then it derives from  $\xi nNL_F\leq \delta$ that $\|\theta-\Pi_{\Theta}(\theta)\|\leq \delta$. Noting the fact that $\partial \chi_{\Theta}(\theta)=\mathrm{N}_\Theta(\theta)$, next we show for $\zeta:=\frac{1}{\xi}(\theta-\Pi_{\Theta}(\theta))$, it holds $\zeta\in \mathrm{N}_\Theta(\Pi_{\Theta}(\theta))$.
In view of the property of the projection operator, one has $(\theta-\Pi_{\Theta}(\theta))^{\T}(\tilde{\theta}-\Pi_{\Theta}(\theta))\leq 0$ for any $\tilde{\theta}\in\Theta$, which leads to $\zeta^{\T}(\tilde{\theta}-\Pi_{\Theta}(\theta))\leq 0$ for any $\tilde{\theta}\in\Theta$. Therefore, it has $\zeta\in \mathrm{N}_\Theta(\Pi_{\Theta}(\theta))$, which completes the proof.
\end{proof}

\section{Proof of Lemma \ref{thm:game}}\label{proof-thm:game}
\begin{proof}
	Note that $x(\tilde{\theta}_k)$ is the NE of the low-level game (\ref{pro:game}) under  $\theta=\tilde{\theta}_k$. Then the proof can be derived following the same line in Theorem 1 and Lemma 2 in \cite{Bianchi21} together with  the definition of the pseudo-gradient mapping $G(x,\theta).$
\end{proof}

\section{Proof of Lemma \ref{lem:conti}}\label{proof:lem:conti}
\begin{proof}
First, prove the Lipschitz continuity of $x(\theta)$. Consider any two vectors $\theta, \theta'\in\R^n$ and the corresponding NEs $x(\theta), x(\theta')\in X$, it follows from (\ref{optcon}) that
\begin{align*}
-G(x(\theta),\theta)\in \mathrm{N}_X(x(\theta)),\ -G(x(\theta'),\theta')\in \mathrm{N}_X(x(\theta')),
\end{align*}
which immediately leads to
\begin{align*}
&\langle-G(x(\theta),\theta), x(\theta')-x(\theta)\rangle\leq 0,\\ &\langle-G(x(\theta'),\theta'), x(\theta)-x(\theta')\rangle\leq 0.
\end{align*}
From the above two inequalities, one has
\begin{align*}
\langle-G(x(\theta),\theta), x(\theta')-x(\theta)\rangle\leq\langle-G(x(\theta'),\theta'), x(\theta')-x(\theta)\rangle,
\end{align*}
and hence
\begin{align*}
&\mu\|x(\theta')-x(\theta)\|^2\notag\\
&\leq \langle G(x(\theta'),\theta)-G(x(\theta),\theta), x(\theta')-x(\theta)\rangle\notag\\
&\leq \langle G(x(\theta'),\theta)-G(x(\theta'),\theta'), x(\theta')-x(\theta)\rangle\notag\\
&\leq \|G(x(\theta'),\theta)-G(x(\theta'),\theta')\|\|x(\theta')-x(\theta)\|\notag\\
&\leq l_\theta\|\theta-\theta'\|\|x(\theta')-x(\theta)\|,
\end{align*}
where the first inequality is due to Assumption \ref{ass:pseudo-gradient}-i) and the last inequality applies  Assumption \ref{ass:pseudo-gradient}-ii).
Therefore, it has
\begin{align}\label{eq:lip-x}
\|x(\theta')-x(\theta)\|\leq \frac{l_\theta}{\mu}\|\theta-\theta'\|.
\end{align}
which verifies the assertion that $x(\theta)$ is Lipschitz continuous.

%Hence, the Lipschitz continuity of $x(\theta)$ holds by Theorem 2B.1, relation (\ref{F-lips-3}), together with the boundness of $\Theta$, that is, there exits some constant $l_x>0$ satisfying
%\begin{align}\label{eq:lip-x}
%\|x(\theta)-x(\theta')\|\leq l_x\|\theta-\theta'\|,\ \forall \theta, \theta'\in\Omega,
%\end{align}}

Subsequently, for any $\theta, \theta'\in\R^n$, one can obtain from the triangle inequality that
\begin{align*}
&|f_i(x(\theta),\theta)-f_i(x(\theta'),\theta')|\notag\\
%&=|f_i(x(\theta),\theta)-f_i(x(\theta'),\theta)+f_i(x(\theta'),\theta)-f_i(x(\theta'),\theta')|\notag\\
&\leq|f_i(x(\theta),\theta)-f_i(x(\theta'),\theta)|+|f_i(x(\theta'),\theta)-f_i(x(\theta'),\theta')|\notag\\
&\leq L_x\|x(\theta)-x(\theta')\|+ L_\theta\|\theta-\theta'\|\notag\\
&\leq (L_x\frac{l_\theta}{\mu}+L_\theta)\|\theta-\theta'\|,
\end{align*}
where the second inequality adopts (\ref{F-lips-2}) and Assumption \ref{ass:set}-iv), and the last inequality is based on (\ref{eq:lip-x}). Consequently, the assertion (\ref{eq:imlip}) holds for $L_F=L_x\frac{l_\theta}{\mu}+L_\theta$.
\end{proof}

\section{Proof of Lemma \ref{lem-inex-zero-grad}}\label{proof-lem-inex-zero-grad}
\begin{proof}
%See Appendix \ref{proof-lem-inex-zero-grad}.
Consider that $\theta_k$ is adapted to $\F_k$ and $u_k$ is independent of $\F_k$, in light of (\ref{lem:gradient-app-exp}) in Lemma \ref{lem:gradient-app}, one has
\begin{align}
&\E[\hat{\nabla}F(\theta_k)+\nabla\hat{\chi}_\Theta(\theta_k)|\F_k]\notag\\
&= \E[\frac{n}{\xi}F(x(\theta_k+\xi u_k),\theta_k+\xi u_k)u_k]+\nabla\hat{\chi}_\Theta(\theta_k)\notag\\
&= \nabla \hat{F}(\theta_k) + \nabla\hat{\chi}_\Theta(\theta_k)\notag\\
&= \nabla \hat{\mathbf{F}}(\theta_k).
\end{align}

It follows from the definitions of $\hat{\nabla}_{\varepsilon_k}F_i(\theta_k)$ and $\hat{\nabla}F_i(\theta_k)$, together with (\ref{F-lips-2}) that
\begin{align}\label{equ:1}
&\|\hat{\nabla}_{\varepsilon_k}F_i(\theta_k)-\hat{\nabla}F_i(\theta_k)\|\notag\\
&=\|\frac{n}{\xi}(f_i(x_{\varepsilon_k}(\theta_k+\xi u_k),\theta_k+\xi u_k)-f_i(x_{\varepsilon_k}(\theta_k),\theta_k))u_k\notag\\
&\quad-\frac{n}{\xi}(f_i(x(\theta_k+\xi u_k),\theta_k+\xi u_k)-f_i(x(\theta_k),\theta_k)u_k\|\notag\\
&\leq \frac{n}{\xi}\|(f_i(x_{\varepsilon_k}(\theta_k+\xi u_k),\theta_k+\xi u_k)\notag\\
&\quad -f_i(x(\theta_k+\xi u_k),\theta_k+\xi u_k))u_k\|\notag\\
&\quad +\frac{n}{\xi}\|(f_i(x_{\varepsilon_k}(\theta_k),\theta_k))-f_i(x(\theta_k),\theta_k))u_k\|\notag\\
&\leq \frac{n}{\xi}L_x\|x_{\varepsilon_k}(\theta_k+\xi u_k)-x(\theta_k+\xi u_k)\|\notag\\
&\quad + \frac{n}{\xi}L_x\|x_{\varepsilon_k}(\theta_k)-x(\theta_k)\|.
\end{align}
Hence, it has
\begin{align}\label{equ:1-1}
&\|\hat{\nabla}_{\varepsilon_k}F_i(\theta_k)-\hat{\nabla}F_i(\theta_k)\|^2\notag\\
&\leq 2\frac{n^2}{\xi^2}L_x^2\|x_{\varepsilon_k}(\theta_k+\xi u_k)-x(\theta_k+\xi u_k)\|^2\notag\\
&\quad + 2\frac{n^2}{\xi^2}L_x^2\|x_{\varepsilon_k}(\theta_k)-x(\theta_k)\|^2\notag\\
&\leq 4\frac{n^2}{\xi^2}L_x^2\varepsilon_k.
\end{align}
Utilizing the $c_r$-inequality $\|\sum_{i=1}^sa_i\|^2\leq s\sum_{i=1}^s\|a_i\|^2$, it obtains that
\begin{align*}
&\E[\|\hat{\nabla}_{\varepsilon_k}F(\theta_k)-\hat{\nabla}F(\theta_k)\|^2|\F_k]\notag\\
&=\E[\|\sum_{i=1}^N(\hat{\nabla}_{\varepsilon_k}F_i(\theta_k)-\hat{\nabla}F_i(\theta_k))\|^2|\F_k]\notag\\
&\leq N\sum_{i=1}^N\E[\|\hat{\nabla}_{\varepsilon_k}F_i(\theta_k)-\hat{\nabla}F_i(\theta_k)\|^2|\F_k]
\end{align*}
which yields (\ref{lem-inex-4}) by combining with (\ref{equ:1-1}).

As for (\ref{lem-inex-2}), note that
\begin{align*}
&\E[\|\hat{\nabla}F(\theta_k)+\nabla\hat{\chi}_\Theta(\theta_k)-\nabla \hat{\mathbf{F}}(\theta_k)\|^2|\F_k]\notag\\
&=\E[\|\hat{\nabla}F(\theta_k)+\nabla\hat{\chi}_\Theta(\theta_k)-\nabla\hat{F}(\theta_k)-\nabla\hat{\chi}_\Theta(\theta_k)|\F_k]\notag\\
&=\E[\|\hat{\nabla}F(\theta_k)\|^2|\F_k]+\|\nabla \hat{F}(\theta_k)\|^2 \notag\\
&\quad-2\E[\hat{\nabla}F(\theta_k)^{\T}\nabla \hat{F}(\theta_k)|\F_k]\notag\\
&=\E[\|\hat{\nabla}F(\theta_k)\|^2|\F_k]-\|\nabla \hat{F}(\theta_k)\|^2,
\end{align*}
where the last equality holds due to
\begin{align*}
\E[\hat{\nabla}F(\theta_k)^{\T}\nabla \hat{F}(\theta_k)|\F_k]&=\E[(\hat{\nabla}F(\theta_k))^{\T} |\F_k]\nabla\hat{F}(\theta_k)\notag\\
&=\|\nabla \hat{F}(\theta_k)\|^2.
\end{align*}
Recall from Lemma \ref{lem:conti} that $F(\theta)$ is $NL_F$-Lipschitz continuous. Then (\ref{lem-inex-2}) holds by virtue of (\ref{222}) in Lemma \ref{lem:gradient-app}. Therefore, the proof is completed.
\end{proof}

\section{Proof of Theorem \ref{thm-main-nonconvex}}\label{proof-thm-main-nonconvex}
\begin{proof}
	Note that $F$ is $NL_F$-Lipschitz continuous, then one has from Lemma \ref{lem:gradient-app} that  $\nabla\hat{F}$ is $\frac{nNL_F}{\xi}$-Lipschitz continuous. Moreover, $\nabla\hat{\chi}_\Theta$ is $\frac{1}{\xi}$-Lipschitz continuous. Hence, $\nabla\hat{\mathbf{F}}$ is $\frac{nNL_F+1}{\xi}$-Lipschitz continuous. In view of the update of $\theta_k$ in (\ref{alg:theta-}), it follows from Lemma 5.7 in \cite{Beck2017} that
	\begin{align*}%\label{thm-1}
		&\hat{\mathbf{F}}(\theta_{k+1})\notag\\
		%&\leq\hat{\mathbf{F}}(\theta_k)+\nabla\hat{\mathbf{F}}(\theta_k)^{\T}(\theta_{k+1}-\theta_k)+\frac{nNL_F+1}{2\xi}\|\theta_{k+1}-\theta_k\|^2\notag\\
		&\leq\hat{\mathbf{F}}(\theta_k)-\alpha \nabla\hat{\mathbf{F}}(\theta_k)^{\T}(\hat{\nabla}_{\varepsilon_k} F(\theta_k)+\nabla\hat{\chi}_\Theta(\theta_k))\notag\\
		&\quad+\frac{nNL_F+1}{2\xi}\alpha^2\|\hat{\nabla}_{\varepsilon_k} F(\theta_k)+\nabla\hat{\chi}_\Theta(\theta_k)\|^2\notag\\
		&=\hat{\mathbf{F}}(\theta_k)-\alpha \nabla\hat{\mathbf{F}}(\theta_k)^{\T}(\hat{\nabla}F(\theta_k)+\nabla\hat{\chi}_\Theta(\theta_k))\notag\\
		&\quad-\alpha \nabla\hat{\mathbf{F}}(\theta_k)^{\T}(\hat{\nabla}_{\varepsilon_k}F(\theta_k)-\hat{\nabla}F(\theta_k))\notag\\
		&\quad+\frac{nNL_F+1}{2\xi}\alpha^2\|\hat{\nabla}F(\theta_k)+\nabla\hat{\chi}_\Theta(\theta_k)\notag\\
		&\quad+\hat{\nabla}_{\varepsilon_k} F(\theta_k)-\hat{\nabla}F(\theta_k)\|^2\notag\\
		&\leq\hat{\mathbf{F}}(\theta_k)-\alpha \nabla\hat{\mathbf{F}}(\theta_k)^{\T}(\hat{\nabla}F(\theta_k)+\nabla\hat{\chi}_\Theta(\theta_k))\notag\\
		&\quad+\frac{\alpha}{4}\| \nabla\hat{\mathbf{F}}(\theta_k)\|^2+\alpha\|\hat{\nabla}_{\varepsilon_k}F(\theta_k)-\hat{\nabla}F(\theta_k)\|^2\notag\\
		&\quad+\frac{nNL_F+1}{\xi}\alpha^2\|\hat{\nabla}F(\theta_k)+\nabla\hat{\chi}_\Theta(\theta_k)\|^2\notag\\ &\quad+\frac{nNL_F+1}{\xi}\alpha^2\|\hat{\nabla}_{\varepsilon_k}F(\theta_k)-\hat{\nabla}F(\theta_k)\|^2,
	\end{align*}
	where the Young's inequality $a_1^Ta_2\leq\frac{\iota}{2}\|a_1\|^2+\frac{1}{2\iota}\|a_2\|^2$ for $\iota=2$ and relation  $\|a_1+a_2\|^2\leq2\|a_1\|^2+2\|a_2\|^2$ have been applied in the last inequality. Taking conditional expectation on both sides of the above inequality with respect to $\F_k$ and exploiting (\ref{lem-inex-3}) in Lemma \ref{lem-inex-zero-grad}, one can obtain
	\begin{align*}%\label{thm-2}
		&\E[\hat{\mathbf{F}}(\theta_{k+1})|\F_k]\notag\\
		&\leq\hat{\mathbf{F}}(\theta_k)-\frac{3\alpha}{4}\| \nabla\hat{\mathbf{F}}(\theta_k)\|^2+2\frac{nNL_F+1}{\xi}\alpha^2\|\nabla \hat{\mathbf{F}}(\theta_k)\|^2\notag\\
		&\quad+2\frac{nNL_F+1}{\xi}\alpha^2\E[\|\hat{\nabla}F(\theta_k)+\nabla\hat{\chi}_\Theta(\theta_k)-\nabla \hat{\mathbf{F}}(\theta_k)\|^2|\F_k]\notag\\
		&\quad+(\frac{nNL_F+1}{\xi}\alpha^2+\alpha)\E[\|\hat{\nabla}_{\varepsilon_k}F(\theta_k)-\hat{\nabla}F(\theta_k)\|^2|\F_k].
	\end{align*}
	Taking expectation on both sides of the aforementioned inequality, then substituting (\ref{lem-inex-2}) and (\ref{lem-inex-4}) in Lemma \ref{lem-inex-zero-grad} into it, and invoking $\alpha\leq\frac{\xi}{4(nNL_F+1)}$, it derives that
	\begin{align*}%\label{thm-3}
		\E[\hat{\mathbf{F}}(\theta_{k+1})]&\leq\E[\hat{\mathbf{F}}(\theta_k)]-\frac{1}{4}\alpha\E[\| \nabla\hat{\mathbf{F}}(\theta_k)\|^2]\notag\\
		&+2N^2n^2L_F^2\frac{nNL_F+1}{\xi}\alpha^2 +\frac{5N^2L_x^2n^2\alpha}{\xi^2}\varepsilon_k.
	\end{align*}
	Taking sum of the above inequality over $k$ from $0$ to $K-1$ and dividing its both sides by $K$, it yields
	\begin{align*}%\label{thm-3}
		&\frac{1}{K}\sum_{k=0}^{K-1}\E[\| \nabla\hat{\mathbf{F}}(\theta_k)\|^2]\notag\\
		&\leq\frac{4(\hat{\mathbf{F}}(\theta_0)-\E[\hat{\mathbf{F}}(\theta_K)])}{\alpha K} + 8N^2n^2L_F^2\alpha\frac{nNL_F+1}{\xi}\notag\\
		&\quad + \frac{20N^2L_x^2n^2\sum_{k=0}^{K-1}\varepsilon_k}{\xi^2K},
	\end{align*}
	which incorporating with the definition of $\hat{\mathbf{F}}^*$ completes the proof.
\end{proof}

\section{Proof of Corollary \ref{corollary-1}}\label{proof-corollary-1}
\begin{proof}
%One just needs to noting the fact that for any $a\in[0,1)$, it has $\sum_{k=0}^K\frac{1}{(k+1)^s}\leq\frac{(K+1)^{1-s}}{1-s}$.
Note that for any $s\in[0,1)$, there holds
\begin{align*}
	\sum_{k=0}^{K-1}\frac{1}{(k+1)^s}&\leq 1+\int_1^{K}\frac{1}{t^s}dt\notag\\
	&=1+\frac{K^{1-s}-1}{1-s}\notag\\
	&\leq\frac{K^{1-s}}{1-s}.
\end{align*}
From $t_k\geq\lceil \frac{-s\ln(k+1)}{\ln q}\rceil$, one can derive   $q^{t_k}\leq \frac{1}{(k+1)^s}$. Hence, it follows from (\ref{defn_error_k}) %(\ref{thm-gamecon}) in Lemma \ref{thm:game} and (\ref{F-lips-1})
that
\begin{align*}
	\sum_{k=0}^{K-1}\varepsilon_k
	%&\leq \sum_{k=0}^{K-1}4NB_X^2q^{t_k}\notag\\
	\leq 2(\|{\bf x}_0\|^2+NB_X^2) \frac{K^{1-s}}{1-s}.
\end{align*}
Then substituting the above inequality into (\ref{thm-nonconvex}) and noting the selection of $\alpha$  immediately lead to assertion (\ref{cor:1}).
\end{proof}

\section{Proof of Corollary \ref{cor-2-}}\label{proof-cor-2-}
\begin{proof}
%See Appendix \ref{cor-2-}.
Applying the $\frac{nNL_F+1}{\xi}$-Lipschitz continuity of $\nabla\hat{\mathbf{F}}$ and substituting $\hat{\nabla}_{\varepsilon_k} F(\theta_k)$ in (\ref{alg:theta-}) by $\hat{\nabla} F(\theta_k)$, one can obtain that
\begin{align}\label{cor-1-}
\hat{\mathbf{F}}(\theta_{k+1})&\leq\hat{\mathbf{F}}(\theta_k)-\alpha \nabla\hat{\mathbf{F}}(\theta_k)^{\T}(\hat{\nabla} F(\theta_k)+\nabla\hat{\chi}_\Theta(\theta_k))\notag\\
&\quad+\frac{nNL_F+1}{2\xi}\alpha^2\|\hat{\nabla}F(\theta_k)+\nabla\hat{\chi}_\Theta(\theta_k)\|^2.
\end{align}
Taking conditional expectation with respect to $\F_k$ on both sides of (\ref{cor-1-}) and relying on (\ref{lem-inex-3}) in Lemma \ref{lem-inex-zero-grad}, one has
\begin{align*}%\label{cor-2-}
&\E[\hat{\mathbf{F}}(\theta_{k+1})|\F_k]\notag\\
&\leq\hat{\mathbf{F}}(\theta_k)-\alpha\| \nabla\hat{\mathbf{F}}(\theta_k)\|^2+\frac{nNL_F+1}{\xi}\alpha^2\|\nabla \hat{\mathbf{F}}(\theta_k)\|^2\notag\\
&\quad+\frac{nNL_F+1}{\xi}\alpha^2\E[\|\hat{\nabla}F(\theta_k)+\nabla\hat{\chi}_\Theta(\theta_k)-\nabla \hat{\mathbf{F}}(\theta_k)\|^2|\F_k].
\end{align*}
Taking expectation on both sides of the preceding inequality, in view of (\ref{lem-inex-2}) in Lemma \ref{lem-inex-zero-grad} and invoking $\alpha\leq\frac{\xi}{2(nNL_F+1)}$, it derives that
\begin{align}\label{cor-3}
\E[\hat{\mathbf{F}}(\theta_{k+1})]&\leq\E[\hat{\mathbf{F}}(\theta_k)]-\frac{1}{2}\alpha\E[\| \nabla\hat{\mathbf{F}}(\theta_k)\|^2]\notag\\
&\quad+N^2n^2L_F^2\frac{nNL_F+1}{\xi}\alpha^2.
\end{align}
For (\ref{cor-3}), summing over $k=0,\ldots,K-1$ and rearranging the inequality, it yields
\begin{align*}%\label{cor-4-}
&\frac{1}{K}\sum_{k=0}^{K-1}\E[\| \nabla\hat{\mathbf{F}}(\theta_k)\|^2]\notag\\
&\leq\frac{2(\hat{\mathbf{F}}(\theta_0)-\E[\hat{\mathbf{F}}(\theta_K)])}{\alpha K} + 2N^2n^2L_F^2\frac{nNL_F+1}{\xi }\alpha.
\end{align*}
The proof is completed by the definition of $\hat{\mathbf{F}}^*$.
\end{proof}

\footnotesize
%\section*{References}

\bibliographystyle{elsarticle-num}
\bibliography{GameNonSOMAS}

@article{barrera2014dynamic,
  title={Dynamic incentives for congestion control},
  author={Barrera, Jorge and Garcia, Alfredo},
  journal={IEEE Transactions on Automatic Control},
  volume={60},
  number={2},
  pages={299--310},
  year={2014},
}

@article{liu2023approximate,
  title={Approximate {N}ash equilibria in large nonconvex aggregative games},
  author={Liu, Kang and Oudjane, Nadia and Wan, Cheng},
  journal={Mathematics of Operations Research},
  volume={48},
  number={3},
  pages={1791--1809},
  year={2023},
}

@inproceedings{zhang2020complexity,
  title={Complexity of finding stationary points of nonconvex nonsmooth functions},
  author={Zhang, Jingzhao and Lin, Hongzhou and Jegelka, Stefanie and Sra, Suvrit and Jadbabaie, Ali},
booktitle = {Proceedings of the 37th International Conference on Machine Learning},
  pages={11173-11182},
  year={2020},
}

@inproceedings{
lin2022gradientfree,
title={Gradient-Free Methods for Deterministic and Stochastic Nonsmooth Nonconvex Optimization},
author={Tianyi Lin and Zeyu Zheng and Michael Jordan},
booktitle={Advances in Neural Information Processing Systems},
editor={Alice H. Oh and Alekh Agarwal and Danielle Belgrave and Kyunghyun Cho},
year={2022},
}

@ARTICLE{may84,
  author={D. Q. Mayne and  E. Polak },
  journal={Journal of Optimization Theory and Applications}, 
  title={Nondifferential optimization via adaptive smoothing}, 
  year={1984},
  volume={43},
  number={4},
  pages={601-613},
}

@article{Hierarchical23,
Author = {Chen, Jiayu and Lan, Tian and Aggarwal, Vaneet},
Title = {Hierarchical Adversarial Inverse Reinforcement Learning},
 journal={IEEE Transactions on Neural Networks and Learning Systems},
  year={2024},
  volume={35},
  number={12},
  pages={17549-17558},
}

@ARTICLE{wang22,
  author={Wang, Yinghui and Geng, Xiaoxue and Chen, Guanpu and Zhao, Wenxiao},
  journal={IEEE Transactions on Neural Networks and Learning Systems}, 
  title={Achieving the Social Optimum in a Nonconvex Cooperative Aggregative Game: {A} Distributed Stochastic Annealing Approach}, 
  year={2025},
 volume={36},
  number={5},
  pages={9709-9716},
}

@article{ghaderi2014opinion,
  title={Opinion dynamics in social networks with stubborn agents: Equilibrium and convergence rate},
  author={Ghaderi, Javad and Srikant, Rayadurgam},
  journal={Automatica},
  volume={50},
  number={12},
  pages={3209--3215},
  year={2014},
}

@article{Nguyen23,
  title={Geometric Convergence of Distributed Heavy-Ball {N}ash Equilibrium Algorithm Over Time-Varying Digraphs With Unconstrained Actions},
  author={Nguyen, Duong Thuy Anh and Nguyen, Duong Tung and Nedi\'{c}, Angelia},
  journal={IEEE Control Systems Letters},
  volume={7},
  number={12},
  pages={1963-1968},
  year={2023},
}

@ARTICLE{FengRR24,
  author={Feng, Zhangcheng and Xu, Wenying and Cao, Jinde},
  journal={IEEE Transactions on Automatic Control}, 
  title={Distributed {N}ash Equilibrium Computation Under {Round-Robin} Scheduling Protocol}, 
  year={2024},
  volume={69},
  number={1},
  pages={339-346},
}

@ARTICLE{Bianchi21,
  author={Bianchi, Mattia and Grammatico, Sergio},
  journal={IEEE Control Systems Letters}, 
  title={Fully Distributed {N}ash Equilibrium Seeking Over Time-Varying Communication Networks With Linear Convergence Rate}, 
  year={2021},
  volume={5},
  number={2},
  pages={499-504},
}

@inproceedings{meng2022linear,
  title={Linear Last-Iterate Convergence for Continuous Games with Coupled Inequality Constraints},
  author={\vspace{0mm}Min Meng and Xiuxian Li},
  Booktitle={2023 62nd IEEE Conference on Decision and Control (CDC)},
  year={2023},
  pages={1076-1081},
  }

@ARTICLE{Gadjov21,
  author={Gadjov, Dian and Pavel, Lacra},
  journal={IEEE Transactions on Automatic Control}, 
  title={Single-Timescale Distributed {GNE} Seeking for Aggregative Games Over Networks via Forward–Backward Operator Splitting}, 
  year={2021},
  volume={66},
  number={7},
  pages={3259-3266},
}

@ARTICLE{meng24online,
  author={Meng, Min and Li, Xiuxian and Chen, Jie},
  journal={IEEE Transactions on Automatic Control}, 
  title={Decentralized {N}ash Equilibria Learning for Online Game With Bandit Feedback}, 
  year={2024},
  volume={69},
  number={6},
  pages={4050-4057},
}

@ARTICLE{Persis20,
  author={De Persis, Claudio and Grammatico, Sergio},
  journal={IEEE Transactions on Automatic Control}, 
  title={Continuous-Time Integral Dynamics for a Class of Aggregative Games With Coupling Constraints}, 
  year={2020},
  volume={65},
  number={5},
  pages={2171-2176},
}

@article{lei22,
author = {Lei, Jinlong and Shanbhag, Uday V.},
title = {Distributed Variable Sample-Size Gradient-Response and Best-Response Schemes for Stochastic {N}ash Equilibrium Problems},
journal = {SIAM Journal on Optimization},
volume = {32},
number = {2},
pages = {573-603},
year = {2022},
}

@INPROCEEDINGS{belgioioso2018,
  author={Belgioioso, Giuseppe and Grammatico, Sergio},
  booktitle={2018 European Control Conference (ECC)}, 
  title={Projected-gradient algorithms for Generalized Equilibrium seeking in Aggregative Games are preconditioned Forward-Backward methods}, 
  year={2018},
   pages={2188-2193},
  }

@article{ Carnevale2022Tracking-based,
 author={Carnevale, Guido and Fabiani, Filippo and Fele, Filiberto and Margellos, Kostas and Notarstefano, Giuseppe},
  journal={IEEE Transactions on Automatic Control}, 
  title={Tracking-Based Distributed Equilibrium Seeking for Aggregative Games}, 
  year={2024},
  volume={69},
  number={9},
  pages={6026-6041},
}

@ARTICLE{Ma14,
  author={Ma, Kai and Hu, Guoqiang and Spanos, Costas J.},
  journal={IEEE Transactions on Control Systems Technology}, 
  title={Distributed Energy Consumption Control via Real-Time Pricing Feedback in Smart Grid}, 
  year={2014},
  volume={22},
  number={5},
  pages={1907-1914},
}

@ARTICLE{Shakarami23,
  author={Shakarami, Mehran and Cherukuri, Ashish and Monshizadeh, Nima},
  journal={IEEE Transactions on Control of Network Systems}, 
  title={Dynamic Interventions With Limited Knowledge in Network Games}, 
  year={2024},
  volume={11},
  number={3},
  pages={1153-1164},
}

@ARTICLE{Nguyen25,
  author={Nguyen, Duong Thuy Anh and Nguyen, Duong Tung and Nedić, Angelia},
  journal={IEEE Transactions on Control of Network Systems}, 
  title={Distributed {N}ash Equilibrium Seeking Over Time-Varying Directed Communication Networks}, 
  year={2025},
  volume={12},
  number={3},
  pages={1917-1929},
}

@article{MENG2023110919,
title = {On the linear convergence of distributed {N}ash equilibrium seeking for multi-cluster games under partial-decision information},
journal = {Automatica},
volume = {151},
pages = {110919},
year = {2023},
author = {Min Meng and Xiuxian Li},
}

@ARTICLE{linon24,
  author={Li, Rongjiang and Chen, Guanpu and Gan, Die and Gu, Haibo and L{\"u}, Jinhu},
  journal={IEEE Transactions on Circuits and Systems I: Regular Papers}, 
  title={{Stackelberg} and {N}ash Equilibrium Computation in Non-Convex Leader-Follower Network Aggregative Games}, 
  year={2024},
  volume={71},
  number={2},
  pages={898-909},
}

@inproceedings{liu22,
 author = {Liu, Boyi and Li, Jiayang and Yang, Zhuoran and Wai, Hoi-To and Hong, Mingyi and Nie, Yu and Wang, Zhaoran},
 booktitle = {Advances in Neural Information Processing Systems},
  pages = {29001-29013},
 title = {Inducing Equilibria via Incentives: Simultaneous Design-and-Play Ensures Global Convergence},
 volume = {35},
 year = {2022}
}

@article{jo2023,
  title={Computing Algorithm for an Equilibrium of the Generalized {Stackelberg} Game},
  author={Jaeyeon Jo, Jihwan Yu and Jinkyoo Park},
  journal={arXiv preprint arXiv:2306.05732},
  year={2023},
}

@INPROCEEDINGS{Maheshwari22,
  author={Maheshwari, Chinmay and Kulkarni, Kshitij and Wu, Manxi and Sastry, S. Shankar},
  booktitle={2022 IEEE 61st Conference on Decision and Control (CDC)}, 
  title={Inducing Social Optimality in Games via Adaptive Incentive Design}, 
  year={2022},
  pages={2864-2869},
  }

@article{Maljkovic2024OnDC,
  title={On decentralized computation of the leader's strategy in bi-level games},
  author={Marko {Maljkovic} and Gustav Nilsson and Nikolas Geroliminis},
journal = {Automatica},
volume = {178},
pages = {112352},
year = {2025},

 }

@article{Maljkovic23,
Author = {Maljkovic, Marko and Nilsson, Gustav and Geroliminis, Nikolas},
Title = {Hierarchical Pricing Game for Balancing the Charging of Ride-Hailing
   Electric Fleets},
Journal = {IEEE Transactions on Control Systems Technology},
Year = {2023},
Volume = {31},
Number = {6},
Pages = {2728-2743},
}

@book{facchinei2003finite,
  title={Finite-Dimensional Variational Inequalities and Complementarity Problems},
  author={Facchinei, Francisco and Pang, Jong-Shi},
  year={2003},
publisher={ Springer-Verlag},
  address={New York},
}

@book{Beck2017,
  title={First-Order Methods in Optimization},
  author={Amir Beck },
  publisher={SIAM},
  address={Philadelphia, PA},
  year={2017},
}

@ARTICLE{gold1977,
  author={A. A. Goldstein},
  journal={Mathematical Programming}, 
  title={Optimization of  {L}ipschitz Continuous Functions}, 
  year={1977},
  volume={13},
  pages={14-22},
}

@ARTICLE{yi2021online,
  author={Yi, Xinlei and Li, Xiuxian and Yang, Tao and Xie, Lihua and Chai, Tianyou and Johansson, Karl Henrik},
  journal={IEEE Transactions on Automatic Control}, 
  title={Distributed Bandit Online Convex Optimization With Time-Varying Coupled Inequality Constraints}, 
  year={2021},
  volume={66},
  number={10},
  pages={4620-4635},
}

@article{ ye16chain,
Author = {Ye, Yu-Sen and Ma, Zu-Jun and Dai, Ying},
Title = {The price of anarchy in competitive reverse supply chains with
   quality-dependent price-only contracts},
Journal = {Transportation Research Part E: Logistics and Transportation Review}, 
Year = {2016},
Volume = {89},
Pages = {86-107},
}

@article{ johari05,
Author = {Johari, R and Mannor, S and Tsitsiklis, JN},
Title = {Efficiency-loss in a network resource allocation game: The case of
   elastic supply},
Journal = {IEEE Transactions on Automatic Control},
Year = {2005},
Volume = {50},
Number = {11},
Pages = {1712-1724},
}

@article{ apt14,
Author = {Apt, Krzysztof R. and Schafer, Guido},
Title = {Selfishness Level of Strategic Games},
Journal = {Journal of Artificial Intelligent Research},
Year = {2014},
Volume = {49},
Pages = {207-240},
}

@article{ Anshelevich08,
Author = {Anshelevich, E. and Dasgupta, A. and Kleinberg, J. and Tardos, E. and Wexler, T. and Roughgarden, T.},
Title = {The price of stability for network design with fair cost allocation},
Journal = {SIAM Journal on Computing},
Year = {2008},
Volume = {38},
number = {4},
Pages = {1602-1623},
}

@article{ Cominetti24,
Author = {Cominetti, Roberto and Dose, Valerio and Scarsini, Marco},
Title = {The price of anarchy in routing games as a function of the demand},
Journal = {Mathematical Progarmming},
Year = {2024},
Volume = {203},
Number = {1-2},
Pages = {531-558},
}

@article{Roughgarden02,
author = {Roughgarden, Tim and Tardos, \'{E}va},
title = {How bad is selfish routing?},
Journal = { Journal of the ACM },
year = {2002},
volume = {49},
number = {2},
pages = {236–259},
}

@article{Galeotti2020,
author = {Galeotti, Andrea and Golub, Benjamin and Goyal, Sanjeev},
title = {Targeting Interventions in Networks},
journal = {Econometrica},
volume = {88},
number = {6},
pages = {2445-2471},
year = {2020},
}

@INPROCEEDINGS{alpcan09,
  author={Alpcan, Tansu and Pavel, Lacra and Stefanovic, Nem},
  booktitle={Proceedings of the 48h IEEE Conference on Decision and Control (CDC) held jointly with 2009 28th Chinese Control Conference}, 
  title={A control theoretic approach to noncooperative game design}, 
  year={2009},
  pages={8575-8580},
}

@inproceedings{Maheshwari2021DynamicTF,
Author = {Maheshwari, Chinmay and Kulkarni, Kshitij and Wu, Manxi and Sastry, S.
   Shankar},
Title = {Dynamic Tolling for Inducing Socially Optimal Traffic Loads},
Booktitle = {2022 American Control Conference (ACC)},
Year = {2022},
Pages = {4601-4607},
}

@article{ratliff21,
  title={Adaptive Incentive Design},
  author={Ratliff, Lillian J. and Fiez, Tanner},
  journal={IEEE Transactions on Automatic Control},
  year={2021},
 volume={66},
  number={8},
  pages={3871-3878},
}

@inproceedings{Qiu2023,
  author = {Qiu, Yuyang and Shanbhag, Uday and Yousefian, Farzad},
 booktitle = {Advances in Neural Information Processing Systems},
 pages = {3425-3438},
 title = {Zeroth-Order Methods for Nondifferentiable, Nonconvex, and Hierarchical Federated Optimization},
 volume = {36},
 year = {2023}
}

@ARTICLE{Chandan24metho,
  author={Chandan, Rahul and Paccagnan, Dario and Marden, Jason R.},
  journal={IEEE Transactions on Automatic Control}, 
  title={Methodologies for Quantifying and Optimizing the Price of Anarchy}, 
  year={2024},
  volume={69},
  number={11},
  pages={7742-7757},
}

@INPROCEEDINGS{Shen2007,
  author={Shen, Hongxia and Ba{\c{s}}ar, Tamer},
  booktitle={Annals of the International Society of Dynamic Games}, 
  title={Incentive-Based Pricing for Network Games with Complete and Incomplete Information}, 
  year={2007},
  volume={9},
  pages={431--458},
}

@article{bacsar2024incentive,
  title={Incentive Designs under Disparate Probabilistic Outlooks Across Decision Makers},
  author={Ba{\c{s}}ar, Tamer},
  journal={IFAC-PapersOnLine},
  volume={58},
  number={30},
  pages={7--12},
  year={2024},
}

@article{sanjari2025incentive,
  title={Incentive designs for {Stackelberg} games with a large number of followers and their mean-field limits},
  author={Sanjari, Sina and Bose, Subhonmesh and Ba{\c{s}}ar, Tamer},
  journal={Dynamic Games and Applications},
  volume={15},
  number={1},
  pages={238--278},
  year={2025},
}

@ARTICLE{pan25bilateral,
  author={Pan, Bangqi and Lu, Jianfeng and Cao, Shuqin and Liu, Jing and Tian, Wenlong and Li, Minglu},
  journal={IEEE Transactions on Mobile Computing}, 
  title={Bilateral Pricing for Dynamic Association in Federated Edge Learning}, 
  year={2025},
  volume={24},
  number={6},
  pages={4684-4697},
}

@ARTICLE{sun25Learning,
  author={Sun, Xiaotian and Xie, Haipeng and Qiu, Dawei and Xiao, Yunpeng and Strbac, Goran and Bie, Zhaohong},
  journal={IEEE Transactions on Smart Grid}, 
  title={Learning the Reluctance of Demand-Side Resources From Equilibrium in Price-Based Demand Response}, 
  year={2025},
  volume={16},
  number={3},
  pages={2699-2702},
}

@ARTICLE{qu24online,
  author={Qu, Zhihai and Li, Xiuxian and Li, Li and Yi, Xinlei},
  journal={IEEE Transactions on Signal Processing}, 
  title={Online Optimization Under Randomly Corrupted Attacks}, 
  year={2024},
  volume={72},
  number={},
  pages={2160-2172},
}

@ARTICLE{mani21quan,
  author={Mani, Ankur and Varshney, Lav R. and Pentland, Alex},
  journal={IEEE Transactions on Signal Processing}, 
  title={Quantization Games on Social Networks and Language Evolution}, 
  year={2021},
  volume={69},
  number={},
  pages={3922-3934},
 }

\end{multicols}
\end{document}